\documentclass{birkau}

\usepackage[all]{xy}
\usepackage{tikz}
\usetikzlibrary{matrix,arrows}
\usepackage[mathscr]{eucal}
\usepackage{amssymb}
\usepackage{enumerate,graphicx}
\usepackage{url} 
\usepackage{multicol}

\usepackage{marginnote}
\numberwithin{equation}{section}

\theoremstyle{plain}

\newtheorem{thm}{Theorem}[section]
\newtheorem{lem}[thm]{Lemma}
\newtheorem{cor}[thm]{Corollary}
\newtheorem{prop}[thm]{Proposition}

\theoremstyle{definition}
\newtheorem{df}[thm]{Definition}
\newtheorem{rem}[thm]{Remark}
\newtheorem{ex}[thm]{Example}
\newtheorem*{prob*}{Problem}

\usepackage{pgf,tikz}

\usetikzlibrary{calc,positioning,shapes,arrows.meta}

\tikzset{%
 shaded/.style={draw, shape=circle, fill=black!35, inner sep=1.4pt},
 unshaded/.style={draw, shape=circle, fill=white, inner sep=1.4pt},
 quasi/.style={draw, shape=rectangle, rounded corners=3pt, fill=white, inner sep=2.5pt, minimum height=14.5pt},
 blob/.style={draw, shape=rectangle, rounded corners=12pt, thin, densely dotted},
 arrow/.style={->, thin, >=latex, shorten >=2.5pt, shorten <=2.5pt},
 order/.style={thin},
 curvy/.style={thin, looseness=1.2, bend angle=70},
 fatcurvy/.style={thin, looseness=1.7, bend angle=75},
 label/.style={shape=rectangle, inner sep=6pt},
 auto}

\font\bmi=cmmi8 scaled 1440
\newcommand{\powerset}{\raise.6ex\hbox{\bmi\char'175 }}

\newcommand{\RRA}{\mathsf{RRA}}
\newcommand{\RDqRA}{\mathsf{RDqRA}}
\newcommand{\DqRA}{\mathsf{DqRA}}
\newcommand{\qRA}{\mathsf{qRA}}
\newcommand{\RA}{\mathsf{RA}}
\newcommand{\RWkRA}{\mathsf{RWkRA}}
\newcommand{\RwkRA}{\mathsf{RwkRA}}
\newcommand{\tdn}{{\triangledown n}}

\newcommand{\littleand}{\mathbin{\wedge\kern -8truept \wedge}}
\newcommand{\bigor}{\mathop{\bigvee\kern -8.5truept \bigvee}}
\newcommand{\Bigor}{\mathop{\bigvee\kern -10truept \bigvee}}
\newcommand{\littleor}{\mathbin{\vee\kern -8truept \vee}}
\renewcommand{\le}{\leqslant}
\renewcommand{\leq}{\leqslant}

\renewcommand{\ge}{\geqslant}
\renewcommand{\preceq}{\preccurlyeq}

\begin{document}


\title[Representable distributive quasi relation algebras]{Representable distributive quasi relation algebras}

\corrauthor[A. Craig]{Andrew Craig}
\address{Department of Mathematics and Applied Mathematics\\
University of Johannesburg\\PO Box 524, Auckland Park, 2006\\South~Africa\\
and\\
National Institute for Theoretical and Computational Sciences (NITheCS)\\Johannesburg, 
South Africa}
\email{acraig@uj.ac.za}

\author[C. Robinson]{Claudette Robinson}
\address{Department of Mathematics and Applied Mathematics\\
University of Johannesburg\\PO Box 524, Auckland Park, 2006\\South~Africa}
\email{claudetter@uj.ac.za}

\subjclass{06F05, 
03B47, 
03G10
}

\keywords{quasi relation algebra, relation algebra, FL-algebra, representability}

\begin{abstract}
We give a definition of representability for distributive quasi relation algebras (DqRAs). These algebras are a generalisation of relation algebras and were first described by Galatos and Jipsen~(2013). Our definition uses a construction that starts with a poset. The algebra is concretely constructed as the lattice of upsets of a partially ordered equivalence relation. The key to defining the three negation-like 
unary operations
is to impose certain symmetry requirements on the partial order. 
Our definition of representable distributive quasi relation algebras is easily seen to be a generalisation of the definition of representable relations algebras by J\'{o}nsson and Tarski (1948). 
We give examples of representable DqRAs and give a necessary condition for an algebra to be finitely representable. 
We leave open the questions of whether every DqRA is representable, and also whether  the class of representable DqRAs forms a variety. Moreover, our definition provides many other  opportunities for investigations in the spirit of those carried out for representable relation algebras. 
\end{abstract}

\maketitle


\section{Introduction}\label{sec:intro}

In 1948 J\'{o}nsson and Tarski~\cite{JT48} defined what are now called representable relation algebras.
These are relation algebras that are isomorphic to a subalgebra of a product of algebras of all binary relations on a set, with the associative binary operation interpreted
as the usual composition of binary relations, its unit as the identity relation, and the involutive
unary operation as the converse relation.
The class of such algebras is denoted by $\mathsf{RRA}$. It can equivalently be defined via a construction using subalgebras of all binary relations in an equivalence relation on a set. Lyndon~\cite{Lyn50} showed that there exist abstract relation algebras (i.e. which satisfied the axioms given by Tarski~\cite{Tar41}) that are not representable. In 1955, Tarski~\cite{Tar55} then showed that $\mathsf{RRA}$ forms a variety, and nine years later Monk~\cite{Monk64} proved that $\RRA$ is not  finitely axiomatizable.

What makes the study of representability 
particularly 
intriguing is the 2001 result of Hirsch and Hodkinson~\cite{HH01} which shows that even for finite relation algebras, the representability problem is undecidable. 

There have been a number of different generalisations of relation algebras proposed by various authors
(cf. non-associative relation algebras by Maddux~\cite{Mad82},  sequential relation algebras by Jipsen and Maddux~\cite{JipMad97}, skew relation algebras by Galatos and Jipsen~\cite{GJ13}).
One of the reasons to try and generalise the class of relation algebras is the fact that the variety of relation algebras does not have a decidable equational theory. 
A distinguishing feature of Galatos and Jipsen's quasi relation algebras~\cite{GJ13} is the fact that they are not required to have a Boolean lattice as their underlying lattice structure. Results  on the decidability for varieties of quasi relation algebras can be found in the paper where they were first described~\cite{GJ13}. 

We provide a construction of concrete distributive quasi relation algebras 
(DqRAs). Perhaps unsurprisingly, this construction begins with a partially ordered set (poset). 
The most general form of the construction considers an equivalence relation $E$ on a poset (Theorem~\ref{thm:main_result}). Specialising to the case that $E=X^2$ (Corollary~\ref{cor:FDqRA}) produces what we call \emph{full distributive quasi relation algebras}.  
(This terminology is in keeping with the tradition for relation algebras and weakening relation algebras~\cite[Section 6]{GJ20-AU}.)
With the construction in place, we 
give two equivalent definitions of a \emph{representable distributive quasi relation algebra} 
(Definition~\ref{def:RDqRA}).
Two definitions are required because our first construction, which uses a partially ordered equivalence relation, is closed under products, while the class of full DqRAs is not. 
Hence our two equivalent definitions mirror those used for the definition of $\mathsf{RRA}$.  

In Section~\ref{sec:qRAs} we present definitions and background on quasi relation algebras. Most of the content there is due to Galatos and Jipsen~\cite{GJ13}, 
but we also give a novel
method for constructing 
new quasi relations algebras from existing ones. We explore when these new algebras will be isomorphic to the original algebras (Proposition~\ref{prop:odd-periodic-iso}). In Section~\ref{sec:construct} we give the details of how to construct a DqRA from a poset satisfying certain symmetry conditions. Using this construction, in Section~\ref{sec:RDqRA} we  present our definition of representable distributive quasi relation algebras.  We denote this class of algebras by $\RDqRA$. 

We give a number of examples of representable DqRAs in
Section~\ref{sec:small-rep}.
There we also 
present one three-element and 
three four-element DqRAs 
for which we cannot find representations and nor can we prove that they are not representable. Indeed, the question of whether there exist non-representable DqRAs remains open. 

Our definition of representable DqRAs immediately gives rise to many interesting questions about distributive quasi relation algebras (and other related classes of algebras). 
We comment on these at the start of Section~\ref{sec:conclusions}
and then end the paper by presenting some open questions. One of these questions asks whether or not $\mathsf{RDqRA}$ forms a variety. 

\section{Quasi relation algebras}\label{sec:qRAs}

In their 2013 paper, Galatos and Jipsen~\cite{GJ13} studied algebras that are a generalisation of relation algebras, which they called \emph{quasi relation algebras} (qRAs). An important feature of this class of algebras, denoted $\qRA$, is that it has a decidable equational theory, whereas $\RA$ does not.  Before we give the definition of qRAs, recall that a \emph{residuated lattice} is an algebra 
$\langle A, \wedge, \vee, \cdot, 1,\backslash,/\rangle$ such that $\langle A, \wedge, \vee\rangle$ is a lattice, $\langle A, \cdot,1\rangle$ is a monoid and for all $a,b \in A$: 
$$a\cdot b \leqslant c \quad  \Longleftrightarrow \quad a \leqslant c/b  \quad \Longleftrightarrow \quad b \leqslant a \backslash c. $$ 

The definition of qRAs builds on that of 
FL-algebras (cf.~\cite[Chapter~2.2]{GJKO}). These algebras are used as algebraic semantics for Full Lambek calculus. 
An FL-algebra $\langle A,\wedge,\vee, \cdot, 1, \backslash,/,0\rangle$ is a residuated lattice with an additional constant $0$. No additional properties are assumed about the $0$; it can be any element of $A$. The residuals and $0$ are used to define two unary operations (known as \emph{linear negations}) as follows: $\sim a = a \backslash 0$ and $-a = 0 / a$. 
Involutive FL-algebras (or InFL-algebras) satisfy \textsf{(In)}, i.e. ${\sim}{-}a={-}{\sim}a=a$ for all $a \in A$. In an FL-algebra, the 
 binary operation $+$ is 
 defined by $a+b = {\sim}(-b\cdot - a)$.  
 
An FL$'$-algebra is an FL-algebra with a unary operation $'$ such that $a''=a$. 
A De Morgan FL$'$-algebra (or DmFL$'$-algebra) is an FL$'$-algebra satisfying $(a\vee b)'=a' \wedge b'$
(\textsf{Dm}). 
A \emph{quasi relation algebra} (qRA) is a DmFL$'$-algebra satisfying:
$$\mathsf{(Di)}\qquad  ({\sim} a)'=-(a')\qquad \qquad \mathsf{(Dp)}\qquad  (a\cdot b)' = a'+b'.$$

The abbreviations (\textsf{\textsf{Di}}) and (\textsf{Dp}) stand for De Morgan involution and De Morgan product, respectively. 

Galatos and Jipsen~\cite[Lemma 2.2]{GJ13} show that  another possible signature 
can be used for  InFL-algebras
(and hence quasi relation algebras). We will write $\mathbf{A}=\langle A, \wedge, \vee, \cdot, 1,0,{\sim},-\rangle$ for an FL-algebra. 

We note that if a De Morgan InFL$'$-algebra satisfies 
\textsf{(Dp)}, then it automatically satisfies \textsf{(Di)} and hence is a qRA (cf.~\cite[Lemma 1]{CJR24}).

Distributive quasi relation algebras (DqRAs) are those qRAs whose underlying lattice is distributive. They will be the focus of this paper from Section~\ref{sec:construct} onwards, but the results in this section apply to all qRAs. 

A qRA (or even an FL-algebra) is \emph{cyclic} if ${\sim} a = {-}a$. 
In Figure~\ref{fig:non-cyclic} we give two examples of non-cyclic qRAs.
Both of these algebras were found with the help of Prover9/Mace4~\cite{P9M4}.
We wish to highlight these as it was the task of being able to 
construct 
\emph{non-cyclic} 
distributive 
quasi relation algebras that proved most challenging.  The algebra on the left is the smallest non-cyclic distributive quasi relation algebra while the algebra on the right is the smallest non-cyclic non-distributive qRA. (There are three other non-cyclic non-distributive qRAs, only one of which is order isomorphic to the one below.) 

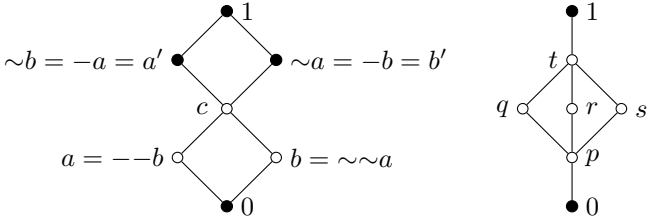
\begin{figure}[ht]
\centering
\begin{tikzpicture}[scale=0.65]
\begin{scope}[xshift=-0.7cm]
  \node[draw,circle,inner sep=1.5pt,fill] (bot) at (0,0) {};
  \node[unshaded] (d) at (-1,1) {};
  \node[unshaded] (e) at (1,1) {};
  \node[unshaded] (c) at (0,2) {};
  \node[draw,circle,inner sep=1.5pt,fill] (a) at (-1,3) {};
  \node[draw,circle,inner sep=1.5pt,fill] (b) at (1,3) {};
  \node[draw,circle,inner sep=1.5pt,fill] (top) at (0,4) {};
  \draw[order] (bot)--(d)--(c)--(a)--(top);
  \draw[order] (bot)--(e)--(c)--(b)--(top);
  \node[label,anchor=east,xshift=1pt] at (a) {${\sim}b=-a=a'$};
  \node[label,anchor=east,xshift=-1pt] at (c) {$c$};
  \node[label,anchor=east,xshift=1pt] at (d) {$a={--}b$};
  \node[label,anchor=west,xshift=-1pt] at (b) {${\sim}a=-b=b'$};
  \node[label,anchor=west,xshift=-1pt] at (e) {$b={\sim\sim}a$};
\node[] at (2.5,0.5) {$={--}a$};
\node[label,anchor=west,xshift=-1pt] at (bot) {$0$};
\node[label,anchor=west,xshift=-1pt] at (top) {$1$};

\end{scope}

\begin{scope}[xshift=8.5cm]
\node[draw,circle,inner sep=1.5pt,fill] (one) at (0,4) {};
\node[unshaded] (t) at (0,3) {};
\node[unshaded] (s) at (1,2) {};
\node[unshaded] (r) at (0,2) {};
\node[unshaded] (q) at (-1,2) {};
\node[unshaded] (p) at (0,1) {};
\node[draw,circle,inner sep=1.5pt,fill] (zero) at (0,0) {};
\draw[order] (zero)--(p)--(q)--(t)--(one);
\draw[order] (p)--(r)--(t)--(s)--(p);
\node[label,anchor=west,xshift=-1pt] at (p) {$p$};
\node[label,anchor=east,xshift=1pt] at (q) {$q$};
\node[label,anchor=west,xshift=-1pt] at (r) {$r$};
\node[label,anchor=west,xshift=-1pt,yshift=-1pt] at (s) {$s$};
\node[label,anchor=east,xshift=1pt] at (t) {$t$};
\node[label,anchor=west,xshift=-1pt] at (one) {$1$};
\node[label,anchor=west,xshift=-1pt] at (zero) {$0$};

\end{scope}

\begin{scope}[yshift=-5cm,xshift=-1.7cm]
\node (somenode) at (1,1){
\begin{tabular}{|c||c|c|c|c|c|  }
\hline $\cdot$ &$a$  &$b$  &$c$ &$a'$ &$b'$  \\\hline
\hline $a$ &  0& 0& 0&$a$ &0  \\
\hline $b$ & 0  &0 &0 &0 &$b$  \\
\hline $c$ &0  & 0&0 &$a$ &$b$  \\
\hline $a'$ & 0 & $b$ &$b$ &$a'$ &$b$  \\
\hline $b'$ &$a$  & 0&$a$ &$a$ &$b'$  \\
\hline
\end{tabular}};
\end{scope}
\begin{scope}[yshift=-5cm,xshift=7.4cm]
\node (somenode) at (1,1){
\begin{tabular}{|c||c|c|c|c|c|c|c|c|  }
\hline $\cdot$ &$p$  &$q$  &$r$ &$s$ &$t$& ${\sim}$ & ${-}$ & $'$  \\\hline
\hline $p$ & 0 &0 &0 &0 &0 &$t$ &$t$ &$t$  \\
\hline $q$ &0 &$p$ &$p$ &0 & $p$&$s$ &$r$ &$q$  \\
\hline $r$ & 0& 0&$p$ &$p$ & $p$&$q$ &$s$ &$s$  \\
\hline $s$ & 0&$p$ &0 &$p$ & $p$&$r$ &$q$ &$r$  \\
\hline $t$ & 0&$p$ &$p$ &$p$ &$p$ &$p$ &$p$ &$p$  \\
\hline
\end{tabular}};

\end{scope}
\end{tikzpicture}
\caption{Two non-cyclic quasi relation algebras.}\label{fig:non-cyclic}
\end{figure}

The following fact will be used frequently in later sections so we state it here as a lemma. 
\begin{lem}\cite[Section 2]{GJ13} \label{lem:dual-iso}
Let 
$\langle A,\wedge,\vee, \cdot, 1, 0, {\sim},{-}\rangle $ 
be an InFL-algebra. Then $\sim$ and $-$ are dual lattice isomorphisms. 
\end{lem}

Given an FL-algebra $\mathbf{A}$ and $a \in A$, we define for $n \in \omega$: 
$$ {\sim}^n a := \underbrace{\sim \sim  \dots \sim}_{n \text{ times}}a \quad \text{and} \quad 
{-}^n a  := \underbrace{- - \dots - }_{n \text{ times}}a. $$ 

The lemma below can be proved using a simple inductive argument. 

\begin{lem}\label{lem:nDi}
Let $\mathbf{A}=\langle A, \wedge, \vee, \cdot, 1,0,{\sim},-,' \rangle$ be an InFL$'$-algebra satisfying
{\upshape (\textsf{Di})}. Then the following hold for any  $n \in \omega$ and $a \in A$:
$$ {\sim}^n\left(a'\right)=\left({-}^n a\right)' \quad \text{and} \quad {-}^n\left(a'\right)=\left({\sim}^n a\right)'.$$ 
\end{lem}

By~\cite[Lemma 4.1]{GJ13}, in a qRA we get: ${-}{\sim} a =a= {\sim}{-}a$, 
and also
${-}0=1={\sim} 0$.
Hence Lemma~\ref{lem:nDi} holds in any qRA, in particular: $(-a)'={\sim} (a')$.

The next lemma will be used in the proof of Propositions~\ref{prop:cyclic-elmts} and~\ref{prop:odd-periodic-iso}. 
As noted by Galatos and Jipsen~\cite[p.5]{GJ13} (see also~\cite[Lemma 3.17]{GJKO}), 
in an InFL-algebra, there are two equivalent 
ways of expressing the $+$ operation: $ a+b= {\sim}(-b\cdot -a)=-({\sim}b\cdot {\sim}a)$.
This fact will be used in the proof below. 

\begin{lem}\label{lem:doubleminuscdot}
Let $\mathbf{A}=\langle A,\wedge, \vee, \cdot, 1, 0,  \sim,-\rangle$ be a InFL-algebra. 
For any 
$n \in \mathbb{Z}^+$
and $a,b \in A$ we have 
$${\sim}^{2n}(a\cdot b)={\sim}^{2n}a\cdot {\sim}^{2n}b \quad\text{and}\quad
{-}^{2n}(a\cdot b)={-}^{2n}a\cdot {-}^{2n}b.
$$
\end{lem}
\begin{proof} 
We proceed by induction
on $n$ for the equality using 
${\sim}^{2n}$. 
The case of 
$n=1$ uses \textsf{(In)} and the definition of $a+b$. We get 
$${\sim}{\sim}(a \cdot b) = {\sim} ({\sim} b + {\sim}a)={\sim}(-({\sim}{\sim}a \cdot {\sim}{\sim}b))={\sim}{\sim}a \cdot {\sim}{\sim}b.$$
Now assume that 
${\sim}^{2k}(a\cdot b)={\sim}^{2k}a\cdot {\sim}^{2k}b$ for some $k\geqslant 1$. It follows that 
${\sim}^{2(k+1)}(a\cdot b)=
{\sim}^{2(k+1)}a\cdot {\sim}^{2(k+1)}b$. 
The proof for ${-}^{2n}(a\cdot b)={-}^{2n}a\cdot {-}^{2n}b$ is similar but will use the other definition of $a+b$ for the case of $n=1$.  
\end{proof}

With Lemma~\ref{lem:doubleminuscdot} in hand, we are able to prove two interesting results which give us new quasi relation algebras from existing ones. 
\begin{prop}\label{prop:cyclic-elmts} 
Let $\mathbf{A}=\langle A, \wedge,\vee, \cdot, 1,0,\sim,-,'\rangle$ be a DiInFL$'$-algebra. The set of elements $C=\{\, a \in A \mid {\sim}a={-}a\,\}$ forms a subuniverse of $\mathbf{A}$ and the resulting subalgebra is a cyclic DiInFL$'$-algebra. 
\end{prop}
\begin{proof}
First we observe that $0,1 \in C$. 
Let $a,b \in C$. We have ${\sim}(a\wedge b) = {\sim}a\vee {\sim} b = {-}a\vee {-}b= {-}(a \wedge b)$. The proof that $a \vee b\in C$ is dual. Now 
$${\sim}(a\cdot b) \,\stackrel{\textsf{(In)}}{=}\, {-}{\sim}{\sim}(a\cdot b) 
\,\stackrel{(\ref{lem:doubleminuscdot})}{=}\, {-}({\sim}{\sim}a \cdot {\sim}{\sim}b)={-}({\sim}{-}a\cdot {\sim}{-}b)={-}(a\cdot b).$$
We use both versions of  \textsf{(Di)} to see that ${\sim}(a')=({-}a)'=({\sim}a)'={-}(a')$. It is clear that the algebra with underlying set $C$ will be cyclic. 
\end{proof} 

The following new result demonstrates that, given a qRA, it is possible to define alternative unary operations that can play the role of the $'$ operation. We will then investigate when these new algebras are isomorphic to the original qRA. 

Let $\mathbf{A}=\langle A,\wedge, \vee, \cdot, 1,0, \sim,-,' \rangle$ be a qRA. Define the following unary operations for $a \in A$ and $n \in \omega$:
$$ 
a^{\triangledown n}:= {\sim}^{2n}(a') \quad\text{and}\quad 
a^{\vartriangle n}:= {-}^{2n}(a').
$$ 

\begin{thm}\label{thm:star-qRA} 
Let $\mathbf{A}=\langle A,\wedge, \vee, \cdot, 1,0, \sim,-,' \rangle$ be a qRA. 
Then the structures $\mathbf{A}^{\triangledown n}=\langle A,\wedge,\vee,\cdot, 1, 0,\sim,-,^{\triangledown n}\rangle $ and
$\mathbf{A}^{\vartriangle n} = \langle A,\wedge,\vee,\cdot, 1,0,\sim,-,^{\vartriangle n}\rangle $ are quasi relation algebras. 
\end{thm}
\begin{proof} We will consider only $\mathbf{A}^{\triangledown n}$. The corresponding results for $\mathbf{A}^{\vartriangle n}$ will follow using the (\textsf{Di}) condition. Our proof proceeds by induction on $n$. When $n=0$ then $a^\tdn=a'$ and so $^\tdn$ is involutive, a dual lattice isomorphism, and (\textsf{Di}) and (\textsf{Dp)} are both satisfied. We assume that these four properties hold for $k \in \omega$. 
Now, using Lemma~\ref{lem:nDi}, we can show that $^{\triangledown(k+1)}$ is involutive: 
\begin{align*}
\left(a^{\triangledown (k+1)}\right)^{\triangledown (k+1)}&=
{\sim}^{2(k+1)}\left[ \left( {\sim}^{2(k+1)}(a')\right)'\right]\\
&=\left[ {-}^{2(k+1)}{\sim}^{2(k+1)}(a')\right]'
\,\stackrel{\textsf{(In)}}{=}\,
a''=a.
\end{align*}
To show that $^{\triangledown(k+1)}$ is a dual lattice isomorphism we use the fact that any even number of applications of $\sim$ 
will be a lattice isomorphism:
\begin{align*}
(a\vee b)^{\triangledown (k+1)}={\sim}^{2(k+1)}((a\vee b)')&= {\sim}^{2(k+1)}(a'\wedge b')\\ &=
{\sim}^{2(k+1)}(a')\wedge {\sim}^{2(k+1)}( b')\\&=
a^{\triangledown (k+1)} \wedge 
b^{\triangledown (k+1)}.
\end{align*}
For (\textsf{Di}), note 
$({\sim} a)^{\triangledown (k+1)}={\sim}^{2(k+1)}(({\sim} a)' )
=(\sim\sim){\sim}^{2k}(({\sim} a)' )
={\sim}{\sim}({\sim} a)^{\triangledown k}$. 
Now, since \textsf{(Di)} and \textsf{(In)} hold for $n=k$ we get 
$$(\sim a)^{\triangledown (k+1)}=(\sim\sim)(-(a^{\triangledown k}))
= - (\sim \sim) (a^{\triangledown k}).$$ 
By the definition of $^{\triangledown k}$, this last expression is equal to $- (\sim \sim) {\sim}^{2k}(a') 
= -{\sim}^{2(k+1)}(a') 
= - (a^{\triangledown (k+1)})$.

Lastly, we show that (\textsf{Dp}) 
holds on $\mathbf{A}^{\triangledown(k+1)}$:
\begin{align*}
(a \cdot b)^{\triangledown (k+1)} = {\sim}^{2(k+1)}((a\cdot b)')
&\stackrel{\textsf{(Dp)}}{=}{\sim}^{2(k+1)}(a'+b')\\
&\,= {\sim}^{2(k+1)}(\sim (-(b')\cdot - (a')))\\
&\stackrel{\textsf{(\ref{lem:doubleminuscdot})}}{=}
{\sim}\left( {\sim}^{2(k+1)}(-(b')) \cdot {\sim}^{2(k+1)} (-(a'))\right)\\
&\stackrel{\textsf{(In)}}{=}
{\sim}\left( -{\sim}^{2(k+1)}(b') \cdot -{\sim}^{2(k+1)}(a')\right)\\
&\,= {\sim}\left( - \left(b^{\triangledown (k+1)} \right)\cdot - \left(a^{\triangledown (k+1)}\right)\right)\\
&\,= a^{\triangledown (k+1)} + b^{\triangledown (k+1)}. & \tag*{\qedhere}
\end{align*}
\end{proof}

Galatos and Jipsen~\cite{GJ12-plog} defined an involutive FL-algebra to be \emph{$n$-periodic} if $n$ is the smallest 
positive integer
such that 
${\sim}^n a={-}^n a$ 
for all $a \in A$. 
Theorem~\ref{thm:star-qRA} showed
that for a qRA $\mathbf{A}$, there are many other qRAs with the same underlying InFL-algebra structure. Below we see that in some cases these ``new'' qRAs are in fact isomorphic to the original algebra.  

\begin{prop}\label{prop:odd-periodic-iso}
Let $\mathbf{A}=\langle A, \wedge, \vee, \cdot, 1,0,\sim,-, '\rangle$ be a quasi relation algebra that is $n$-periodic for some odd natural number $n$. 
\begin{enumerate}[\normalfont (i)]
\item If $m\leqslant n$ and $m$ is odd, then $\mathbf{A} \cong \mathbf{A}^{\triangledown m}$ via the map  
$a \mapsto {-}^{n-m}a$ and $\mathbf{A}\cong \mathbf{A}^{\vartriangle m}$ via the map $a \mapsto {\sim}^{n-m}a$.  
\item If $m>n$ and $m$ is odd, then $\mathbf{A} \cong \mathbf{A}^{\triangledown m}$ via the map  
$a \mapsto {\sim}^{m-n}a$ and $\mathbf{A}\cong \mathbf{A}^{\vartriangle m}$ via the map $a \mapsto {-}^{m-n}a$.  
\end{enumerate}
\end{prop}
\begin{proof}
Assume $\mathbf{A}$ is $n$-periodic for 
$n$ odd, 
and let $m$ be odd with $m \leqslant n$. Defining $h(a):={-}^{n-m}a$ we see that since $-$ is a dual lattice isomorphism and since $h$ is an even number of applications of $-$, it will be a lattice isomorphism. Further, by Lemma~\ref{lem:doubleminuscdot}, $h$ will preserve the monoid operation. 
Clearly $h$ preserves the constants 
and hence also both $-$ and ${\sim}$.
Lastly, 
\begin{align*}
(h(a))^{\triangledown m} = 
{\sim}^{2m}\left( \left[ h(a) \right]'\right) &\;=
{\sim}^{2m}\left[\left({-}^{n-m}a\right)'\right]\\
&\stackrel{(\ref{lem:nDi})}{=} 
{\sim}^{2m}{\sim}^{n-m}(a') \\ 
&\;= {\sim}^n{\sim}^m (a')= {-}^n{\sim}^m(a') = {-}^{n-m}(a')= h(a').
\end{align*}

The proof that $\mathbf{A}\cong \mathbf{A}^{\vartriangle m}$ in part (i) is similar. For part (ii), we only prove the fact that $h(a)={\sim}^{m-n}a$ preserves the $'$ operation: 
\begin{align*}
\left(h(a)\right)^{\triangledown m}  = {\sim}^{2m}\left[h(a)\right]' & = {\sim}^{2m}\left[\left({\sim}^{m-n}a\right)'\right]\\
&\stackrel{(\ref{lem:nDi})}{=} {\sim}^{2m}{-}^{m-n}(a') \\
& = {\sim}^{m+ n}(a')\\
& = {\sim}^{m -n + n + n}(a')\\
& = {\sim}^{m- n}{\sim}^n{\sim}^n(a')\\
&= {\sim}^{m- n}{\sim}^n {-}^n (a')  = {\sim}^{m-n}(a') = h(a'). && \qedhere
\end{align*}
\end{proof}

The non-distributive qRA on the right of Figure~\ref{fig:non-cyclic} is an example of an odd-periodic qRA to which Proposition~\ref{prop:odd-periodic-iso} can be applied. The 2-periodic (distributive) qRA on the left of Figure~\ref{fig:non-cyclic} is 
an even-periodic qRA for which $\mathbf{A}\ncong \mathbf{A}^{\triangledown 1}$ 
(this can be checked by considering the four possible lattice isomorphisms). 
Using Mace4~\cite{P9M4},  we have 
found even periodic quasi relation algebras for which $\mathbf{A}\cong \mathbf{A}^{\triangledown 1}$. We discuss this further in Section~\ref{sec:conclusions}.

\section{Concrete distributive quasi relation algebras}\label{sec:construct} 

Before we give the construction of a distributive quasi relation algebra from a poset, we recall some definitions and results about binary relations.
We will also prove a few lemmas that will be used later to prove the main results. 

\subsection{Binary relations}

We will use the following special binary relations on a set $X$: the \emph{empty relation} $\varnothing$, 
the \emph{identity relation} (\emph{diagonal relation}) $\text{id}_X = \left\{\,\left(x, x\right) \mid x \in X\,\right\}$,
and the \emph{universal relation} $X^2$. 
In what follows we will write $R^c$ to denote the complement of $R$ in a binary relation $U\subseteq X^2$.  
The \emph{converse of $R$} is 
$R^\smile = \left\{\,\left(x, y\right) \mid \left(y, x\right) \in R\,\right\}$
and the \emph{composition} of $R, S \subseteq X^2$ is 
$R\mathbin{;} S = \left\{\,\left(x, y\right) \mid \left(\exists z \in X\right)\left(\left(x, z\right) \in R \textnormal{ and } \left(z, y\right) \in S\right)\,\right\}$.
We also define $R^0 = \text{id}_X$ and $R^{n+1} = R^n\mathbin{;} R$ for $n \in \omega$.

\begin{prop}\label{prop:properties_binary_relations}
Let $R, S, T, U$ be binary relations on a set $X$ such that $R, S, T \subseteq U$. Further, let $I$ be an index set and for each $i\in I$, assume $R_i \subseteq U$. Then the following hold:
\begin{multicols}{2}
\begin{enumerate}[\normalfont (i)]
\item $\left(R^\smile\right)^\smile = R$
\item $\left(R^\smile\right)^c = \left(R^c\right)^\smile$
\item $\left(\displaystyle \bigcup_{i\in I} R_i\right)^\smile = \displaystyle\bigcup_{i\in I}R_i^\smile$ 
\item $\left(\displaystyle \bigcap_{i\in I} R_i\right)^\smile = \displaystyle\bigcap_{i\in I}R_i^\smile$ 
\item $\mathrm{id}_X\mathbin{;} R = R\mathbin{;} \mathrm{id}_X = R$
\item $\left(R\mathbin{;} S\right)\mathbin{;}T = R\mathbin{;} \left(S\mathbin{;} T\right)$
\item $\left(R\mathbin{;} S\right)^\smile = S^\smile\mathbin{;} R^\smile$
\item $\displaystyle \bigcup_{i \in I}R_i\mathbin{;} S = \bigcup_{i\in I}\left(R_i\mathbin{;} S\right)$
\item $\displaystyle S\mathbin{;} \bigcup_{i\in I}R_i = \bigcup_{i\in I}\left(S\mathbin{;} R_i\right)$
\end{enumerate}
\end{multicols}
\end{prop}

We are interested in algebras whose universes are binary relations. The assumption that $E$ is an equivalence relation together with the lemma below will ensure that if $R, S \subseteq E$, then $R\mathbin{;} S\subseteq E$ and $R^\smile \subseteq E$.  

\begin{lem}\label{lem:comp_and_conv_in_E}
Let $R, S, T\subseteq X^2$. 
\begin{enumerate}[\normalfont (i)]
\item If $T$ is transitive and $R, S \subseteq T$, then $R\mathbin{;} S \subseteq T$.
\item If $T$ is symmetric and $R \subseteq T$, then $R^{\smile} \subseteq T$. 	
\end{enumerate}	
\end{lem}

The following lemma will be used in Section \ref{sec:RDqRA}: 

\begin{lem}\label{lem:(gammaR)capE=gamma(RcapE)}
Let $X$ be a set and $T$ a symmetric and transitive relation on $X$. If $S$ is a binary relation on $X$ such that $S \subseteq T$, then the following hold for all $R\subseteq T$:
\begin{enumerate}[\normalfont (i)]
\item $\left(S\mathbin{;} R\right) \cap T = S\mathbin{;}\left(R \cap T\right)$
\item $\left(R\mathbin{;} S\right) \cap T = \left(R \cap T\right)\mathbin{;}S$
\end{enumerate}		
\end{lem}

\begin{proof}
For the left-to-right inclusion of item (i), let $(x, y) \in \left(S\mathbin{;} R\right) \cap T$. 
Then $(x, y) \in S\mathbin{;} R$ and $(x, y) \in T$. The former implies that there is some $z \in X$ such that $(x, z) \in S$ and $(z, y) \in R$. Hence, we have $(z,x) \in S^\smile$. Since $S \subseteq T$ and $T$ is symmetric, $S^\smile \subseteq T$ by Lemma \ref{lem:comp_and_conv_in_E}, and consequently $(z, x) \in T$.  
Thus, using the fact that $T$ is transitive, we can conclude that $(z, y) \in T$. 
Therefore
$(z, y) \in R \cap T$ and so $(x, y) \in S\mathbin{;} \left(R \cap T\right)$. 

Now, let  $(x, y) \in S\mathbin{;} \left(R \cap T\right)$. Then there is some $z \in X$ such that $(x, z) \in S$ and $(z, y) \in R \cap T$. The second part implies $(z, y) \in R$ and $(z, y) \in T$. Hence, $(x, y) \in S\mathbin{;} R$, and since $S\subseteq T$ and $T$ is transitive, we also have $(x, y) \in T$. We thus get $(x, y) \in \left(S\mathbin{;} R\right) \cap T$. 
	
The proof of item (ii) is similar.  
\end{proof}

The \emph{graph} of a function $\gamma: X \to X$ is the binary relation $G_\gamma \subseteq X^2$ defined by
$G_\gamma = \left\{\,\left(x, \gamma(x)\right)\mid x \in X\,\right\}$. 
In what follows we will use $\gamma$ interchangeably to denote either the function $\gamma: X \to X$ or the binary relation that is its graph. 

Let $\gamma: X \to X$ be a function and 
$n \in\omega$. Then we define 
$\gamma^0=\text{id}_X$ 
and $\gamma^{n+1} = \gamma^n\circ \gamma=\gamma \mathbin{;} \gamma^n$. The \emph{order} of a function $\gamma: X \to X$, denoted $\left|\gamma\right|$, is the smallest positive integer $n$ such that $\gamma^n = \text{id}_X$. 
If no such integer exists then we say that $\gamma$ has infinite order. 

The following lemma will be very 
useful and it applies in particular when $\gamma$ is a bijective function from $X$ to $X$.
We remind the reader that, as in Proposition~\ref{prop:properties_binary_relations}, $R^c$ is the complement of $R$ in $U$. 

\begin{lem}\label{lem:important_eq_injective_map}
Let $X$ be a set and $U$ a symmetric and transitive relation on $X$. Further, let $R, S, \gamma \subseteq U$. If 
$\gamma$ satisfies 
$\gamma^{\smile}\mathbin{;} \gamma = \mathrm{id}_X$ and $\gamma \mathbin{;} \gamma^{\smile}=\mathrm{id}_X$
then the following hold:
\begin{enumerate}[\normalfont (i)]
\item If $\left|\gamma \right| = n\ge 2$, then $\gamma^{n-k} = \left(\gamma^\smile\right)^{k}$, where $1 \le k \le n-1$.
\item $\left(\gamma\mathbin{;} R\right)^c = \gamma \mathbin{;} R^c$
\item $\left(R\mathbin{;} \gamma\right)^c = R^c \mathbin{;} \gamma$
\item $\gamma \mathbin{;} \left(R\cap S\right) = \left(\gamma\mathbin{;} R\right) \cap \left(\gamma \mathbin{;} S\right)$
\item $\left(R \cap S\right) \mathbin{;} \gamma = \left(R \mathbin{;} \gamma\right) \cap \left(S \mathbin{;} \gamma\right)$
\end{enumerate}	
\end{lem}

\begin{proof}
We will prove items (ii) and (iv). For item (ii), first let $\left(x, y\right) \in \left(\gamma\mathbin{;} R\right)^c$. Then $(x, y) \in U$ and $(x, y) \notin \gamma\mathbin{;} R$. The second part implies that for all $z \in X$, we have $\left(x, z\right) \notin \gamma$ or $\left(z, y\right) \notin R$. Since $\mathrm{id}_X = \gamma\mathbin{;} \gamma^\smile$ and $\left(x, x\right) \in \mathrm{id}_X$, we have $\left(x, u\right) \in \gamma$ for some $u \in X$. Hence, it must be the case that $\left(u, y\right) \notin R$. Now $\gamma \subseteq U$, and so, since $U$ is symmetric and transitive, we get $\left(u, y\right) \in U$. It follows that $\left(u, y\right) \in R^c$, and therefore $\left(x, y\right) \in \gamma\mathbin{;} R^c$. 

For the reverse inclusion, let $\left(x, y\right) \in \gamma\mathbin{;} R^c$. Then there is some $z \in X$ such that $\left(x, z\right) \in \gamma$ and $\left(z, y\right) \in R^c$. The second part implies that $\left(z, y\right) \in U$ but $\left(z, y\right) \notin R$. We have to show that $\left(x, y\right) \in U$ while $\left(x, y\right) \notin \gamma\mathbin{;} R$. Since $\gamma \subseteq U$ and $U$ is transitive, we get $\left(x, y\right) \in U$. To see that $\left(x, y\right) \notin \gamma\mathbin{;} R$, let $u$ be an arbitrary element of $X$ such that $\left(x, u\right) \in \gamma$. Now $\left(x, z\right) \in \gamma$, so  $\left(z, x\right) \in \gamma^\smile$, and therefore $\left(z, u\right) \in \gamma^\smile\mathbin{;} \gamma = \mathrm{id}_X$. This shows that $z = u$, and hence $\left(u, y\right) \notin R$. It thus follows that $\left(x, y\right) \notin \gamma\mathbin{;} R$, and consequently, since $(x,y) \in U$, we get  $\left(x, y\right) \in \left(\gamma\mathbin{;} R\right)^c$. 

For item (iv), first let $(x, y) \in \gamma\mathbin{;} \left(R \cap S\right)$. Then there is some $z \in X$ such that $(x, z) \in \gamma$ and $(z, y) \in R \cap S$. The second part implies $(z, y) \in R$ and $(z, y) \in S$. Hence, we get $(x, y) \in \gamma\mathbin{;} R$ and $(x, y) \in \gamma\mathbin{;}S$, which gives $(x, y) \in \left(\gamma \mathbin{;} R\right) \cap \left(\gamma \mathbin{;} S\right)$.

Now let $(x, y) \in \left(\gamma\mathbin{;} R\right) \cap \left(\gamma\mathbin{;} S\right)$.  
This means there are $u, v \in X$ such that $(x, u) \in \gamma$, $(u, y) \in R$, $(x, v) \in \gamma$ and $(v, y) \in S$. From the first part we get $(u, x) \in \gamma^\smile$, and consequently $(u, v) \in \gamma^{\smile}\mathbin{;} \gamma = \mathrm{id}_X$. Hence, $(u, y) \in R$ and $(u, y) \in S$, and therefore $(x, y) \in \gamma\mathbin{;} \left(R \cap S\right)$. 
\end{proof}

\subsection{Constructing distributive quasi relation algebras from posets}\label{subsec:construct}

Our construction proceeds in two steps: first we show how to construct a distributive InFL-algebra. 
Then we identify the right definition of a $'$ operation in order to get a distributive quasi relation algebra. 

Before outlining our construction, we note that the underlying lattice structure that we obtain is the same as that used by Galatos and Jipsen~\cite{GJ20-AU, GJ20-ramics}, and more recently by Jipsen and \v{S}emrl~\cite{JS23}. The major difference is that we place additional symmetry requirements on the 
poset, which allows us to define a  
$0$ which is not equal to $(\leqslant^c)^{\smile}$. 
This allows for representations of algebras with $0=1$ and also non-cyclic InFL-algebras.

As is the case for relation algebras, we will construct both so-called equivalence DqRAs and full DqRAs. Our proofs cover the case of equivalence DqRAs and then the case of full DqRAs will follow naturally (Corollary~\ref{cor:FDqRA}).

Let $\mathbf X = \left(X, \le\right)$ be a poset  and $E$ an equivalence relation on $X$ such that ${\le} \subseteq E$. The equivalence relation $E$ can be partially ordered as follows: for all $(u, v), (x, y) \in E$,
$$(u, v) \preceq (x, y) \qquad  \textnormal{iff} \qquad x \le u \textnormal{ and } v\le y.$$
Then $\mathbf E = \left(E, \preceq\right)$ is a poset and the set of up-sets of $\mathbf E$, $\textsf{Up}\left(\mathbf E\right)$, ordered by inclusion, is a distributive lattice.

The next lemma shows that composition, converse and complement can be combined in various ways to yield 
up-sets/down-sets
when applied to up-sets/down-sets of $\mathbf{E}$.
Furthermore it shows that when the 
graph of an order automorphism/dual order automorphism 
of $\mathbf{X}$ is composed with up-sets/down-sets then  up-sets are produced. 
Below we denote by $\mathsf{Down}(\mathbf{E})$ the down-sets of the poset $\mathbf{E}$.

In the rest of this paper, when we write $R^c$ for $R \subseteq E$, we intend the complement to be taken in the equivalence relation $E$. 

\begin{lem}\label{lem:important_up-_and_down-sets}
Let $\mathbf X = \left(X, \le\right)$ be a poset and $E$ an equivalence relation on $X$ such that ${\le} \subseteq E$. Furthermore, let  $\alpha: X \rightarrow X$ be an order automorphism of $\mathbf X$ such that $\alpha \subseteq E$ and $\beta: X \rightarrow X$ a dual order automorphism of $\mathbf X$ such that $\beta \subseteq E$.
\begin{enumerate}[\normalfont (i)]
\item If $R, S  \in \mathsf{Up}\left(\mathbf E\right)$, then $R\mathbin{;} S \in \mathsf{Up}\left(\mathbf E\right)$.
\item If $R, S  \in \mathsf{Down}\left(\mathbf E\right)$, then $R\mathbin{;} S \in \mathsf{Down}\left(\mathbf E\right)$.
\item $R \in \mathsf{Up}\left(\mathbf E\right)$ iff $R^c \in \mathsf{Down}\left(\mathbf E\right)$.
\item $R \in \mathsf{Up}\left(\mathbf E\right)$ iff $R^{\smile} \in \mathsf{Down}\left(\mathbf E\right)$.
\item If $R \in \mathsf{Up}\left(\mathbf E\right)$, then $\alpha\mathbin{;} R \in \mathsf{Up}\left(\mathbf E\right)$ and $R\mathbin{;} \alpha \in \mathsf{Up}\left(\mathbf E\right)$. 
\item If $R \in \mathsf{Down}\left(\mathbf E\right)$, then $\beta\mathbin{;}R\mathbin{;}\beta \in \mathsf{Up}\left(\mathbf E\right)$.
\end{enumerate}	
\end{lem}

\begin{proof}
For item (i), assume $R, S \in \mathsf{Up}\left(\mathbf E\right)$. 
Let $(u, v) \in R\mathbin{;} S$ and $(x, y) \in E$, and assume $(u, v) \preceq (x, y)$. 
Since $(u, v) \in R\mathbin{;} S$, there is some $w \in X$ such that $(u, w) \in R$ and $(w,v) \in S$. 
From the assumption we have $x \leq u$ and $v \leq y$. 
But $w \leq w$, so $(u, w) \preceq (x, w)$, and hence, since $R$ is an up-set, $(x, w) \in R$. 
Likewise, $w \leq w$ and $v \leq y$ imply that $(w, v)\preceq (w, y)$. 
Therefore, since $(w, v) \in S$ and $S$ is an up-set, it follows that $(w, y) \in S$.
This means $(x, w) \in R$ and $(w, y) \in  S$, so $(x, y) \in R\mathbin{;} S$. 

The proof of item (ii) is similar to that of item (i). 

Item (iii) is a well-known fact. 
For item (iv), assume $R \in \mathsf{Up}\left(\mathbf E\right)$. Let $(x, y) \in R^\smile$ and $(u, v) \in E$, and assume $(u, v) \preceq (x, y)$. Then $(y, x) \in R$, $x \le u$ and $v \le y$. We thus have $(y, x)\preceq (v, u)$, and so, since $R$ is an upset and $(y, x) \in R$, it follows that $(v, u) \in R$. Hence, $(u, v) \in R^\smile$. 

Conversely, assume $R^{\smile} \in \mathsf{Down}\left(\mathbf E\right)$. Let $(u, v) \in R$ and $(x, y) \in E$ such that $(u, v)\preceq (x, y)$. The latter implies that $x \le u$ and $v \le y$, and so $(y, x)\preceq (v, u)$. Hence, since $(u, v) \in R$ implies that $(v, u) \in R^\smile$ and $R^\smile$ is a down-set, we get $(y, x) \in R^{\smile}$. Thus, $(x, y) \in R$.

For item (v), let $R \in \mathsf{Up}\left(\mathbf E\right)$. Since $\alpha \subseteq E$ and $E$ is transitive, it follows that $\alpha\mathbin{;} R \subseteq E$ and $R\mathbin{;}\alpha \subseteq E$. We now show that $\alpha\mathbf{;} R$ is an up-set. Let $(x, y) \in E$ and $(u, v) \in  \alpha\mathbin{;} R$. Assume $(u, v) \preceq (x, y)$. Since $(u, v) \in  \alpha \mathbin{;}R$, we have $(\alpha(u), v) \in R$.  From $(u, v) \preceq (x, y)$ we get $x \le u$ and $v \le y$. The former together with the fact that $\alpha$ is an order automorphism imply that $\alpha(x) \le \alpha(u)$. Hence, $(\alpha(u), v) \preceq (\alpha(x), y)$. Consequently, since  $(\alpha(u), v) \in R$ and $R$ is an up-set, we get $(\alpha(x), y) \in R$, which means $(x, y) \in \alpha \mathbin{;} R$. 
Since  $\alpha^{-1}$ is  an order automorphism, one can show that $R\mathbin{;}\alpha$ is also an up-set. 

Finally, for item (vi), let $R \in \mathsf{Down}\left(\mathbf E\right)$. As before, the assumption that $\beta \subseteq E$ guarantees that $\beta \mathbin{;}R\mathbin{;} \beta \subseteq E$. To see that $\beta \mathbin{;} R\mathbin{;} \beta$ is an up-set, let $(x, y) \in E$ and $(u, v) \in  \beta \mathbin{;} R\mathbin{;}\beta$, and assume $(u, v) \preceq (x, y)$. Then we get $\left(\beta(u), \beta^{-1}(v)\right) \in R$, $x \le u$ and $v \le y$. But $\beta$ is a dual order automorphism of $\mathbf X$, so we have $\beta(u) \le \beta(x)$ and $\beta^{-1}(y) \le \beta^{-1}(v)$. Hence, it follows that $\left(\beta(x), \beta^{-1}(y)\right)\preceq \left(\beta(u), \beta^{-1}(v)\right)$, and so, since $\left(\beta(u), \beta^{-1}(v)\right) \in R$ and $R$ is a down-set, we get $\left(\beta(x), \beta^{-1}(y)\right)\in R$, which means $(x, y) \in \beta \mathbin{;} R\mathbin{;}\beta$. \qedhere
\end{proof}

\begin{lem}\label{lem:order_identity}
Let $\mathbf X = \left(X, \le\right)$ be a poset and $E$ an equivalence relation on $X$ such that ${\le} \subseteq E$.
\begin{enumerate}[\normalfont (i)]
\item ${\le} \in \mathsf{Up}\left(\mathbf E\right)$.
\item ${\le}\mathbin{;} R = R\mathbin{;} {\le} = R$ for all $R \in \mathsf{Up}\left(\mathbf E\right)$. 
\item ${\le}^{\smile}\mathbin{;} S = S\mathbin{;} {\le}^{\smile} = S$ for all $S \in \mathsf{Down}\left(\mathbf E\right)$.
\end{enumerate}	
\end{lem}

\begin{proof}
It follows from assumption that ${\le} \subseteq E$. To show that $\le$ is an up-set, let $(x, y) \in E$, and assume $u \le v$ and $(u, v)\preceq (x, y)$. 
We get $x \leqslant u \leqslant v \leqslant y$. 
For item (ii), let $R \in \mathsf{Up}\left(\mathbf{E}\right)$. 
We will prove that ${\le}\mathbin{;}  R = R$.
To prove the left-to-right inclusion, let $(x, y) \in  {\le}\mathbin{;} R$. 
Then there is some $z \in X$ such that $x \le z$ and $(z, y) \in R$. 
Since 
$x \leq z$, we have $(z, y)\preceq (x, y)$. 
Hence, as $R$ is an up-set, $(x, y) \in R$.

For the right-to-left inclusion, let $(x, y) \in R$. 
From $(x,x) \in {\le}$ we get $(x,y) \in {\le}\mathbin{;} R$. 
In a similar way one can show that $R\mathbin{;}  {\le} = R$

For (iii), let $S \in \mathsf{Down}\left(\mathbf E\right)$. Then $S^\smile \in \mathsf{Up}\left(\mathbf E\right)$ by Lemma \ref{lem:important_up-_and_down-sets}. Hence,  ${\le}^\smile\mathbin{;} S = \left(S^\smile\mathbin{;} {\le}\right)^\smile = \left(S^\smile\right)^\smile = S$ by (ii). Likewise,  $S\mathbin{;} {\le}^{\smile} = S$.
\end{proof}

Consider a poset $\mathbf X = \left(X, \le\right)$ and an equivalence relation $E$ on $X$ such that ${\le} \subseteq E$. 
Since composition is associative and $\le$ is the identity of composition according to Lemma \ref{lem:order_identity}, the structure $\left\langle\mathsf{Up}\left(\mathbf E\right), \mathbin{;}, \le\right\rangle$ is a monoid. 
Furthermore, composition is residuated and the residuals $\backslash, /$ are defined by the usual expressions for residuals on binary relations: $R\backslash S = (R^{\smile}\mathbin{;} S^c)^c$ and $R/S = (R^c\mathbin{;} S^{\smile})^c$. The structure 
$\left\langle \mathsf{Up}\left(\mathbf E\right), \cap, \cup, \mathbin{;}, \leqslant, \backslash, / \right\rangle$ is therefore a residuated (distributive) lattice. 

Recall that a qRA is an algebra of the form $\left\langle A, \wedge, \vee, \cdot, 1, 0, {\sim}, {-}, '\right\rangle$. In order to extend $\left\langle \mathsf{Up}\left(\mathbf E\right), \cap, \cup, \mathbin{;},\le, \backslash, / \right\rangle$ to a quasi relation algebra, we need to find a suitable $0$ and $'$. 
The $0$ must be chosen to guarantee an InFL-algebra, and $'$ must be involutive and satisfy \textsf{(Dm)},
\textsf{(Di)} 
and \textsf{(Dp)}.  
Before we give our choice for $0$ and define $'$, we give examples of the construction so far. 

\begin{figure}
\begin{tikzpicture}[scale=0.09]
		
\begin{scope}
\node at (0,-10) {\makebox[0pt]{\smash{$\left(X, \leqslant\right)$}}};
\node[unshaded] (A) at (0,7) {};
\node[unshaded] (B) at (0,20) {};
\draw[order] (A) to (B);
\node[anchor=east,xshift=-1mm] at (A) {$x$};
\node[anchor=east,xshift=-1mm] at (B) {$y$};
\end{scope}
		
\begin{scope}[xshift=25cm]
\node at (0,-10) {\makebox[0pt]{\smash{$\left(X^2,\preceq\right)$}}};
\node[unshaded] (A) at (0,10) {};
\node[unshaded] (B) at ($(A)+(135:10)$) {};
\node[unshaded] (C) at ($(A)+(45:10)$) {};
\node[unshaded] (D) at ($(B)+(45:10)$) {};
\draw[order] (A) to (B);
\draw[order] (A) to (C);
\draw[order] (B) to (D);
\draw[order] (C) to (D);
\node[anchor=north] at (A) {$\left(y,x\right)$};
\node[anchor=east] at (B) {$\left(x,x\right)$};
\node[anchor=west] at (C) {$\left(y,y\right)$};
\node[anchor=south] at (D) {$\left(x,y\right)$};
\end{scope}
		
\begin{scope}[xshift=78cm]
\node at (0,-10) {\makebox[0pt]{\smash{$(\mathsf{Up}\left(\left(X^2,\preceq\right)\right),\subseteq)$}}};
\node[unshaded] (A) at (0,0) {};
\node[unshaded] (C) at ($(A)+(90:10)$) {};
\node[unshaded] (D) at ($(C)+(135:10)$) {};
\node[unshaded] (E) at ($(C)+(45:10)$) {};
\node[unshaded] (F) at ($(D)+(45:10)$) {};
\node[unshaded] (G) at ($(F)+(90:10)$) {};
\draw[order] (A) to (C);
\draw[order] (C) to (D);
\draw[order] (C) to (E);
\draw[order] (D) to (F);
\draw[order] (E) to (F);
\draw[order] (F) to (G);
\node[anchor=west,xshift=1mm,yshift=-1mm] at (A) {$\varnothing$};
\node[anchor=west,xshift=1mm,yshift=-1mm] at (C) {$\{\left(x,y\right)\}$};
\node[anchor=east] at (D) {$\{\left(x,x\right), \left(x,y\right)\}$};
\node[anchor=west] at (E) {$\{\left(y,y\right), \left(x,y\right)\}$};
\node[anchor=west,xshift=1mm,yshift=1mm] at (F) {$\leqslant$};
\node[anchor=west,xshift=1mm,yshift=-1mm] at (G) {$X^2$};
\end{scope}
\end{tikzpicture}
\caption{A distributive lattice of binary relations from a poset. }
\label{fig:ex_construction_6elements}
\end{figure}
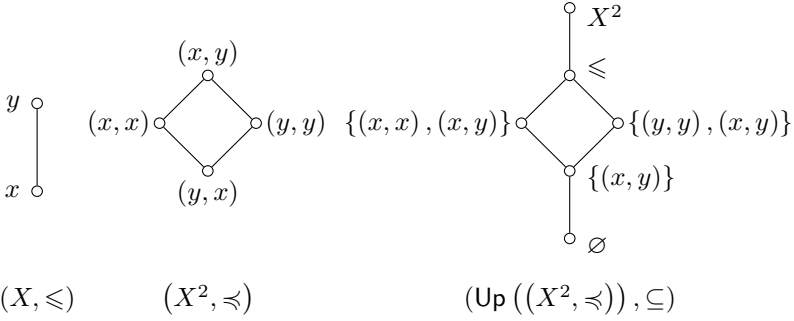

\begin{ex}\label{ex:two-elmt-chain}
Let $X = \left\{x, y\right\}$ and ${\le} = \left\{(x,x), (x,y), (y,y)\right\}$. Take $E = X^2$. Consider the  
poset $\left(X^2, \preceq\right)$,  and the lattice $\mathsf{Up}((X^2,\preceq))$ in Figure \ref{fig:ex_construction_6elements}.
Since $-1=0= {\sim} 1$, and $\sim$ and $-$ are dual lattice isomorphisms  according to Lemma \ref{lem:dual-iso}, 
the only option for $0$ is 
$0=\{\left(x,y\right)\}=\left(\leqslant^c\right)^\smile$. 
Hence, we get
${\sim}\{\left(x,x\right),\left(x,y\right)\} =
\{\left(y,y\right),\left(x,y\right)\}= -\,\{\left(x,x\right),\left(x,y\right)\}$
and 
${\sim} \{\left(y,y\right),\left(x,y\right)\}  =
\{\left(x,x\right),\left(x,y\right)\} =
-\{\left(y,y\right),\left(x,y\right)\}$,
and so
we obtain a cyclic InFL-algebra. 
\end{ex}

\begin{ex}\label{ex:Vposet}
Consider the posets $\left(X, \le\right)$ and $\left(X^2, \preceq\right)$ in Figure \ref{fig:ex_construction_50elements}. 
The lattice of up-sets of $(X^2,\preceq)$ has 50 elements. Given that the height of the lattice is 9, and the height of $\leqslant$ is 5, the height of $0={-}{\leqslant}={\sim}{\le}$ must be 4. Hence 
there are ten possibilities for $0$. However, by making a few calculations, it can be shown that eight of these possibilities yield a $-$ and $\sim$ that are not dual lattice isomorphisms and can therefore be eliminated. The only possibilities  are $0 = {\uparrow} \left(y,z\right) \cup\, {\uparrow} \left(z,y\right) = \left(\le^c\right)^\smile$ and $0 = {\uparrow} \left(y,y\right) \cup\, {\uparrow} \left(z,z\right)$. In the former case we obtain a cyclic
InFL-algebra and in the latter 
a non-cyclic InFL-algebra. Notice that ${\uparrow} \left(y,y\right) \cup\, {\uparrow} \left(z,z\right) = \left(\le^c\right)^\smile\mathbin{;} \left\{\left(x,x\right), \left(y,z\right), \left(z,y\right)\right\}$ and that $\left\{\left(x,x\right), \left(y,z\right), \left(z,y\right)\right\}$ is the graph of the order automorphism $\alpha: X\to X$ defined by $\alpha(x) = x$, $\alpha (y) = z$ and $\alpha(z) = y$. 
\end{ex}

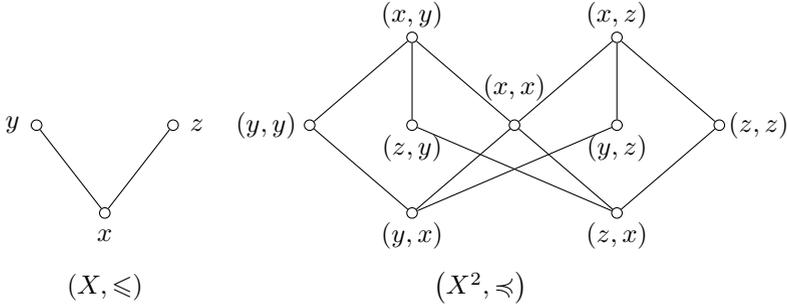
\begin{figure}[t]
\begin{tikzpicture}[scale=0.09]
		
\begin{scope}
\node at (0,-5) {\makebox[0pt]{\smash{$\left(X, \leqslant\right)$}}};
\node[unshaded] (A) at (0,7) {};
\node[unshaded] (B) at (-10,20) {};
\node[unshaded] (C) at (10,20) {};
\draw[order] (A) to (B);
\draw[order] (A) to (C);
\node[anchor=north,yshift=-1mm] at (A) {$x$};
\node[anchor=east,xshift=-1mm] at (B) {$y$};
\node[anchor=west,xshift=1mm] at (C) {$z$};
\end{scope}
		
\begin{scope}[xshift=55cm]
\node at (0,-5) {\makebox[0pt]{\smash{$\left(X^2,\preceq\right)$}}};
\node[unshaded] (A) at (-10,7) {};
\node[unshaded] (B) at (20,7) {};
\node[unshaded] (C) at (-25,20) {};
\node[unshaded] (D) at (-10,20) {};
\node[unshaded] (E) at (-10,33) {};
\node[unshaded] (F) at (5,20) {};
\node[unshaded] (G) at (35,20) {};
\node[unshaded] (H) at (20,20) {};
\node[unshaded] (I) at (20,33) {};
\draw[order] (A) to (C);
\draw[order] (C) to (E);
\draw[order] (A) to (F);
\draw[order] (F) to (E);
\draw[order] (B) to (F);
\draw[order] (B) to (G);
\draw[order] (F) to (I);
\draw[order] (G) to (I);
\draw[order] (A) to (H);
\draw[order] (B) to (D);
\draw[order] (D) to (E);
\draw[order] (H) to (I);
\node[anchor=north] at (A) {$\left(y,x\right)$};
\node[anchor=north] at (B) {$\left(z,x\right)$};
\node[anchor=east,xshift=-0.5mm] at (C) {$\left(y,y\right)$};
\node[anchor=north] at (D) {$\left(z,y\right)$};
\node[anchor=south] at (E) {$\left(x,y\right)$};
\node[anchor=west] at (G) {$\left(z,z\right)$};
\node[anchor=north] at (H) {$\left(y,z\right)$};
\node[anchor=south] at (I) {$\left(x,z\right)$};
\node[anchor=south,yshift=2mm] at (F) {$\left(x,x\right)$};
\end{scope}
\end{tikzpicture}
\caption{The twisted product from a three-element poset.}
\label{fig:ex_construction_50elements}
\end{figure}

Example~\ref{ex:Vposet} suggests that a non-identity order automorphism on the poset could be used to find a candidate other than $(\le^c)^\smile$ for $0$. 
The next lemma shows that if $\alpha$ is an order automorphism of a poset $\left(X, \le\right)$ and $E$ is an equivalence relation containing the order and $\alpha$, then $\alpha\mathbin{;} \left(\le^c\right)^\smile$ is an up-set of $\mathbf{E}$ and, in fact, $\alpha\mathbin{;} \left(\le^c\right)^\smile = \left(\le^c\right)^\smile\mathbin{;} \alpha$.

\begin{lem}\label{lem:definition_of_0}
Let $\mathbf X = \left(X, \le\right)$ be a poset and $E$ an equivalence relation on $X$ such that ${\le} \subseteq E$. If $\alpha: X \rightarrow X$ is an order automorphism of $\mathbf X$ such that $\alpha \subseteq E$, then $\alpha\mathbin{;} \left(\le^{c}\right)^{\smile} = \left(\le^{c}\right)^{\smile}\mathbin{;} \alpha \in \mathsf{Up}\left(\mathbf E\right)$.
\end{lem}

\begin{proof}
It follows from Lemma \ref{lem:important_up-_and_down-sets} that both  $\alpha\mathbin{;} \left(\le^{c}\right)^{\smile}$ and $\left(\le^{c}\right)^{\smile} \mathbin{;}  \alpha$ belong to $\mathsf{Up}\left(\mathbf E\right)$, so all we have to check is that $\alpha \mathbin{;} \left({\le}^{c}\right)^{\smile} = \left({\le}^{c}\right)^{\smile}\mathbin{;}  \alpha$. To this end, let $(x, y) \in \alpha\mathbin{;} \left({\le}^{c}\right)^{\smile}$.
Then we have $(\alpha(x), y) \in \left({\le}^{c}\right)^{\smile}$, which implies that $y \not\le \alpha(x)$. 
Since $\alpha^{-1}$ is an order automorphism, 
$\alpha^{-1}(y) \not\le \alpha^{-1}\left(\alpha(x)\right)=x$.  
Hence 
$(x, \alpha^{-1}(y)) \in \left(\le^{c}\right)^{\smile}$
and so $(x, y) \in \left(\le^{c}\right)^{\smile}\mathbin{;} \alpha$. 

For the reverse inclusion, let $(x, y) \in \left(\le^{c}\right)^{\smile}\mathbin{;} \alpha$. Then we have \break $(x, \alpha^{-1}(y)) \in \left(\le^{c}\right)^{\smile}$, and so $\alpha^{-1}\left(y\right) \not\le x$. But since $\alpha$ is an order automorphism, we get $y = \alpha\left(\alpha^{-1}\left(y\right)\right) \not\le \alpha(x)$. Hence, $(\alpha(x), y) \in \left(\le^{c}\right)^{\smile}$, which gives $(x, y) \in \alpha\mathbin{;}\left(\le^{c}\right)^{\smile}$. 
\end{proof}

We will now add $0 = \alpha\mathbin{;}\left(\le^{c}\right)^{\smile} = \left(\le^{c}\right)^{\smile}\mathbin{;} \alpha$ to the signature of the residuated distributive  lattice $\left\langle \mathsf{Up}\left(\mathbf E\right), \cap, \cup, \mathbin{;}, \leqslant, \backslash, / \right\rangle$  to obtain the (distributive) FL-algebra
$\left\langle \mathsf{Up}\left(\mathbf E\right), \cap, \cup, \mathbin{;}, \le, \backslash, /, 0\right\rangle$. 
The next lemma gives an alternative way of calculating ${\sim} R$ and $- R$ if we set $0 = \alpha\mathbin{;} \left(\le^{c}\right)^{\smile} = \left(\le^{c}\right)^{\smile}\mathbin{;} \alpha$. 

\begin{lem}\label{lem:twiddle_minus_alternative}
Let $\mathbf X = \left(X, \le\right)$ be a poset and $E$ an equivalence relation on $X$ such that ${\le} \subseteq E$
and let $\alpha: X \rightarrow X$ be an order automorphism of $\mathbf X$ such that $\alpha \subseteq E$. 
If  
$0 = \alpha\mathbin{;} \left(\le^{c}\right)^{\smile} = \left(\le^{c}\right)^{\smile}\mathbin{;} \alpha$ 
then 
${\sim} R = \left(R^c\right)^{\smile}\mathbin{;}\alpha$ and $-R = \alpha \mathbin{;} \left(R^{c}\right)^{\smile}$ for all $R \in \mathsf{Up}\left(\mathbf E\right)$.
\end{lem}

\begin{proof}
Assume $0 = \alpha \mathbin{;}\left(\le^{c}\right)^{\smile}  = \left(\le^{c}\right)^{\smile} \mathbin{;} \alpha$ and let $R\in \mathsf{Up}(\mathbf E)$. Then:
\begin{align*}
{\sim} R = R\backslash 0  = \left(R^\smile \mathbin{;} 0^c\right)^c
& = \left(R^\smile \mathbin{;} \left(\left(\le^{c}\right)^{\smile} \mathbin{;} \alpha\right)^c\right)^c \\
& = \left(R^\smile \mathbin{;} \left(\left(\le^c\right)^\smile\right)^c \mathbin{;} \alpha \right)^c & \left(\textnormal{Lemma \ref{lem:important_eq_injective_map}}\right)\\
& = \left(R^\smile\mathbin{;} {\le}^\smile\mathbin{;} \alpha \right)^c &  \left(\textnormal{Proposition \ref{prop:properties_binary_relations}}\right)\\
& = \left(R^\smile \mathbin{;} \alpha\right)^c & \left(\textnormal{Lemma \ref{lem:order_identity}}\right)\\
& = \left(R^c\right)^\smile \mathbin{;}\alpha & \left(\textnormal{Lemma \ref{lem:important_eq_injective_map}}\right) 
\end{align*}
and 
\begin{align*}
- R = 0 / R  = \left(0^c \mathbin{;} R^\smile\right)^c
& = \left(\left(\alpha \mathbin{;}\left(\le^{c}\right)^{\smile}\right)^c\mathbin{;} R^\smile\right)^c \\
& = \left(\alpha\mathbin{;}\left(\left(\le^c\right)^\smile\right)^c \mathbin{;} R^\smile \right)^c & \left(\textnormal{Lemma \ref{lem:important_eq_injective_map}}\right)\\
& = \left(\alpha\mathbin{;} {\le}^\smile \mathbin{;} R^\smile \right)^c &  \left(\textnormal{Proposition \ref{prop:properties_binary_relations}}\right)\\
& = \left(\alpha \mathbin{;} R^\smile\right)^c & \left(\textnormal{Lemma \ref{lem:order_identity}}\right)\\
& = \alpha \mathbin{;} \left(R^c\right)^\smile & \left(\textnormal{Lemma \ref{lem:important_eq_injective_map}} \right)  
\end{align*}
\end{proof}

The lemma above can be generalised as follows: 
\begin{lem}\label{lem:general_twiddle_and _minus}
Let $\mathbf{X}=\left(X,\le\right)$ be a poset and $E$ an equivalence relation on $X$ such that ${\le} \subseteq E$.  Let $\alpha: X \to X$ be an order automorphism of $\mathbf X$ such that $\alpha \subseteq E$. 
Set  $0 = \alpha \mathbin{;} (\le^c)^\smile = (\le^c)^\smile \mathbin{;} \alpha$. Then the following hold for all $R \in \mathsf{Up}\left(\mathbf E\right)$:
\begin{enumerate}[\normalfont (i)]
\item ${\sim}^n R = \left(\alpha^\smile\right)^{\frac{n-1}{2}}\mathbin{;} \left(R^c\right)^\smile\mathbin{;} \alpha^{\frac{n+1}{2}}$ for all odd $n \ge 1$. 
\item ${\sim}^n R = \left(\alpha^\smile\right)^{\frac{n}{2}} \mathbin{;} R\mathbin{;} \alpha^{\frac{n}{2}}$ for all even $n \ge 2$. 
\item ${-}^n R = \alpha^{\frac{n+1}{2}}\mathbin{;} \left(R^c\right)^\smile\mathbin{;} \left(\alpha^\smile\right)^{\frac{n-1}{2}}$ for all odd $n \ge 1$. 
\item ${-}^n R = \alpha^{\frac{n}{2}}\mathbin{;} R\mathbin{;} \left(\alpha^\smile\right)^{\frac{n}{2}}$ for all even $n \ge 2$. 
\end{enumerate}	
\end{lem}

\begin{proof}
The proof is by induction on $n$. We will prove item (i) and leave the rest
to the reader.
The base case follows from Lemma \ref{lem:twiddle_minus_alternative}. Assume the statement holds 
for some 
$n = 2k + 1, k \ge 0$. We now show that the statement holds for $n = 2k + 3$. By assumption we get
\begin{align*}
{\sim}^{2k+3} R 
& = {\sim}{\sim}\left(\sim\right)^{2k + 1}R\\
& = {\sim}{\sim} \left(\left(\alpha^\smile\right)^{\frac{2k+1 -1}{2}}\, ; \left(R^c\right)^\smile\mathbin{;} \alpha^{\frac{2k+ 1+1}{2}}\right)\\ 
& = {\sim}{\sim} \left(\left(\alpha^\smile\right)^{k}\mathbin{;} \left(R^c\right)^\smile\mathbin{;} \alpha^{\frac{2k+ 2}{2}}\right).
\end{align*}
Hence 
\begin{align*}
{\sim}^{2k+3} R 
& = {\sim}\left[\left( \left(\left(\alpha^\smile\right)^{k}\mathbin{;} \left(R^c\right)^\smile\mathbin{;} \alpha^{\frac{2k+ 2}{2}}\right)^c\right)^\smile\mathbin{;} \alpha\right] & \left(\textnormal{Lemma \ref{lem:twiddle_minus_alternative}}\right)\\
& = {\sim}\left[\left( \left(\alpha^\smile\right)^{k}\mathbin{;} \left(\left(R^c\right)^\smile\right)^c\mathbin{;} \alpha^{\frac{2k+ 2}{2}}\right)^\smile\mathbin{;} \alpha\right] & \left(\textnormal{Lemma \ref{lem:important_eq_injective_map}}\right)\\
& = {\sim}\left[\left( \left(\alpha^\smile\right)^{k}\mathbin{;} R^\smile\mathbin{;} \alpha^{\frac{2k+ 2}{2}}\right)^\smile\mathbin{;} \alpha\right] & \left(\textnormal{Proposition \ref{prop:properties_binary_relations}}\right)\\
& = {\sim}\left[ \left(\alpha^\smile\right)^{\frac{2k+2}{2}}\mathbin{;} R\mathbin{;} \alpha^k\mathbin{;} \alpha\right] & \left(\textnormal{Proposition \ref{prop:properties_binary_relations}}\right)\\
& = {\sim}\left[ \left(\alpha^\smile\right)^{\frac{2k+2}{2}}\mathbin{;} R\mathbin{;} \alpha^{k+1}\right] \\
& = \left(\left(\left(\alpha^\smile\right)^{\frac{2k+2}{2}}\mathbin{;} R\mathbin{;} \alpha^{k+1}\right)^c\right)^\smile \mathbin{;} \alpha & \left(\textnormal{Lemma \ref{lem:twiddle_minus_alternative}}\right) \\
& = \left(\left(\alpha^\smile\right)^{\frac{2k+2}{2}}\mathbin{;} R^c\mathbin{;} \alpha^{k+1}\right)^\smile \mathbin{;}\alpha & \left(\textnormal{Lemma \ref{lem:important_eq_injective_map}}\right)\\
& = \left(\alpha^\smile\right)^{k+1}\mathbin{;} \left(R^c\right)^\smile\mathbin{;} \alpha^{\frac{2k+2}{2}}\mathbin{;} \alpha & \left(\textnormal{Proposition \ref{prop:properties_binary_relations}}\right)\\
& = \left(\alpha^\smile\right)^{\frac{2k+3 -1}{2}}\mathbin{;} \left(R^c\right)^\smile\mathbin{;}\alpha^{\frac{2k+3 + 1}{2}}. & & \qedhere
\end{align*}
\end{proof}

Lemmas~\ref{lem:twiddle_minus_alternative} and~\ref{lem:general_twiddle_and _minus} can now be used to prove the theorem below. 

\begin{thm}\label{thm:InFL-algebra}
Let $\mathbf X = \left(X, \le\right)$ be a poset and $E$ an equivalence relation on $X$ such that ${\le} \subseteq E$
and let $\alpha: X \rightarrow X$ be an order automorphism of $\mathbf X$ such that $\alpha \subseteq E$. 
\begin{enumerate}[\normalfont (i)]
\item If\, $0=\alpha \mathbin{;} (\leqslant^c)^\smile=(\leqslant^c)^\smile \mathbin{;} \alpha$, then 
$\langle \mathsf{Up}(\mathbf{E}), \cap, \cup, \mathbin{;}, \leqslant, 0, {\sim},{-}  \rangle$ is a distributive InFL-algebra. 
\item The algebra in {\upshape (i)} is 
$n$-periodic if and only if\, $|\alpha|=n$. 
\end{enumerate}
\end{thm}
 
\begin{proof}
Let $R\in\mathsf{Up}(\mathbf{E})$. By
Lemmas~\ref{lem:important_eq_injective_map} and \ref{lem:twiddle_minus_alternative} 
we have 
$${\sim}{-}R =  \left(\alpha \mathbin{ ;} (R^c)^\smile\right)^{c\smile} \mathbin{;} \alpha = \left( \alpha\mathbin{;} R^\smile\right)^\smile \mathbin{;} \alpha = (R^\smile)^\smile \mathbin{;} \alpha^\smile \mathbin{;} \alpha = R.
$$ 
A similar calculation shows that $-{\sim}R=R$, proving (i). For (ii), we must consider two cases. The first case is for $n$ odd. 
Assume that $|\alpha|=n$. Then  
\begin{align*}
{-}^n R & = \alpha^{\frac{n+1}{2}}\mathbin{;} \left(R^c\right)^\smile \mathbin{;} \left(\alpha^\smile\right)^{\frac{n-1}{2}} & \left(\textnormal{Lemma \ref{lem:general_twiddle_and _minus}}\right)\\
& = \left(\alpha^\smile\right)^{\frac{n-1}{2}}\mathbin{;} \left(R^c\right)^\smile\mathbin{;} \alpha^{\frac{n+1}{2}} & \left(\textnormal{Lemma \ref{lem:important_eq_injective_map}}\right)\\
& = {\sim}^n R & \left(\textnormal{Lemma \ref{lem:general_twiddle_and _minus}}\right)
\end{align*}
Now assume $|\alpha| \neq n$. That is, for some $x \in X$, we have $\alpha^n(x)\neq x$. Let $y = \alpha^n(x)$. 
We know that $(y,y)\preceq (x,x)$ if and only if $x \leqslant y$ and $y \leqslant x$. Since $y\neq x$ we have $(y,y) \in \left({\downarrow}\{(x,x)\}\right)^c$, which is an up-set of~$\mathbf E$. 
By Lemma~\ref{lem:general_twiddle_and _minus}(i),
${\sim}^n\left({\downarrow}\{(x,x)\}\right)^c
=(\alpha^\smile)^{\frac{n-1}{2}}\mathbin{;} \left({\downarrow}\{(x,x)\}\right)^\smile \mathbin{;} \alpha^{\frac{n+1}{2}}$. Since $y=\alpha^n(x)$ and $n$ is odd, it follows that
$(\alpha^{-1})^{\frac{n-1}{2}}(y) =\alpha^{\frac{n+1}{2}}(x)$. Now from 
$(x,x)\in \left({\downarrow}\{(x,x)\}\right)^{\!\smile}$ we get that  
$$\left( \alpha^{\frac{n-1}{2}}(x), \alpha^{\frac{n+1}{2}}(x)\right) \in 
(\alpha^\smile)^{\frac{n-1}{2}}\mathbin{;} \left({\downarrow}\{(x,x)\}\right)^{\!\smile} \mathbin{;} \alpha^{\frac{n+1}{2}}.$$ That is, $\left( \alpha^{\frac{n-1}{2}}(x),(\alpha^{-1})^{\frac{n-1}{2}}(y)\right)\in {\sim}^n \left({\downarrow}\{(x,x)\}\right)^c$. Now suppose we have $\left( \alpha^{\frac{n-1}{2}}(x),(\alpha^{-1})^{\frac{n-1}{2}}(y)\right)\in {-}^n\left({\downarrow}\{(x,x)\}\right)^c = \alpha^{\frac{n+1}{2}}\mathbin{;} \left({\downarrow}\{(x,x)\}\right)^\smile  \mathbin{;} \left(\alpha^\smile\right)^{\frac{n-1}{2}} $. This implies $\left(\alpha^{\frac{n+1}{2}}(\alpha^{\frac{n-1}{2}}(x)),y\right)=(y,y)\in {\downarrow}\{(x,x)\}$, a contradiction. 
Hence ${\sim}^n \left({\downarrow}\{(x,x)\}\right)^c \neq {-}^n\left({\downarrow}\{(x,x)\}\right)^c$.

If we assume $n$ is even and that $|\alpha|=n$, then:
\begin{align*}
{-}^n R & = \alpha^{\frac{n}{2}}\mathbin{;} R \mathbin{;} \left(\alpha^\smile\right)^{\frac{n}{2}} & \left(\textnormal{Lemma \ref{lem:general_twiddle_and _minus}}\right)\\
& = \left(\alpha^\smile\right)^{\frac{n}{2}}\mathbin{;} R\mathbin{;} \alpha^{\frac{n}{2}} & \left(\textnormal{Lemma \ref{lem:important_eq_injective_map}}\right)\\
& = {\sim}^n R & \left(\textnormal{Lemma \ref{lem:general_twiddle_and _minus}}\right)
\end{align*} 
For the converse, assume that $|\alpha|\neq n$. Hence there exists $x$ such that $\alpha^n(x)\neq x$. Let $y=\alpha^n(x)$. Since $y\neq x$ we have
$(y,y) \notin {\uparrow} \{(x,x)\}$. Now, ${\sim}^n \left( {\uparrow} \{(x,x)\}\right) = (\alpha^\smile)^{\frac{n}{2}} \mathbin{;} 
\left( {\uparrow} \{(x,x)\}\right) \mathbin{;} \alpha^{\frac{n}{2}}$. Since $(x,x)\in  {\uparrow} \{(x,x)\}$ we get $\left( \alpha^{\frac{n}{2}}(x), \alpha^{\frac{n}{2}}(x)\right) \in {\sim}^n \left({\uparrow}\{(x,x)\}\right)$. If   
$\left( \alpha^{\frac{n}{2}}(x), \alpha^{\frac{n}{2}}(x)\right) \in {-}^n \left({\uparrow}\{(x,x)\}\right)$, then $\left( (\alpha^{-1})^{\frac{n}{2}}(y), (\alpha^{-1})^{\frac{n}{2}}(y)\right) \in {-}^n \left({\uparrow}\{(x,x)\}\right)=\alpha^{\frac{n}{2}} \mathbin{;} \left( {\uparrow} \{(x,x)\}\right) \mathbin{;} (\alpha^{\smile})^{\frac{n}{2}}$.
This contradicts $(y,y) \notin {\uparrow}\{(x,x)\}$. Thus  ${\sim}^n\left({\uparrow}\{(x,x)\}\right)\neq {-}^n\left({\uparrow}\{(x,x)\}\right)$.
\end{proof}

The final step in constructing concrete DqRAs is to define a $'$ operation on $\mathsf{Up}(\mathbf{E})$ that will be involutive, and satisfies \textsf{(Dm)},
\textsf{(Di)} 
and \textsf{(Dp)}. First, we look at an example. 

\begin{ex}\label{ex:prime-for-6elmt}
Consider the posets $\left(X, \le\right)$ and $\left(X^2, \preceq\right)$ in Figure \ref{fig:ex_construction_6elements} again. Since $'$ is an involutive dual lattice isomorphism, it must be the case that either $\left\{\left(x,x\right), \left(x,y\right)\right\}' = \left\{\left(y,y\right), \left(x,y\right)\right\}$ and $\left\{\left(y,y\right), \left(x,y\right)\right\}' = \left\{\left(x,x\right), \left(x,y\right)\right\}$ or $\left\{\left(x,x\right), \left(x,y\right)\right\}' = \left\{\left(x,x\right), \left(x,y\right)\right\}$ and $\left\{\left(y,y\right), \left(x,y\right)\right\}' = \left\{\left(y,y\right), \left(x,y\right)\right\}$.  
In the first case, we get 
$\left(\left\{\left(x,x\right), \left(x,y\right)\right\}\mathbin{;} 
\left\{\left(y,y\right), \left(x,y\right)\right\}\right)' = \left\{\left(x,y\right)\right\}' = {\le}$ and $\left\{\left(x,x\right), \left(x,y\right)\right\}' + \left\{\left(y,y\right), \left(x,y\right)\right\}' = X^2$, which means \textsf{(Dp)} does not hold. However, it can be shown that 
\textsf{(Dp)} holds in the second case and that the resulting algebra is a distributive quasi relation algebra. Notice that 
$\left\{\left(x,x\right), \left(x,y\right)\right\}'  = \left\{\left(x,y\right), \left(y,x\right)\right\}\mathbin{;} \left\{\left(x,x\right), \left(x,y\right)\right\}^c\mathbin{;} \left\{\left(x,y\right), \left(y,x\right)\right\}$, so we get $\left\{\left(x,x\right), \left(x,y\right)\right\}'
 =\left\{\left(x,x\right), \left(x,y\right)\right\}$. Also, $'$ fixes $\left\{\left(y,y\right), \left(x,y\right)\right\}$ since
$\left\{\left(y,y\right), \left(x,y\right)\right\}' =\left\{\left(x,y\right), \left(y,x\right)\right\}\mathbin{;} \left\{\left(y,y\right), \left(x,y\right)\right\}^c\mathbin{;} \left\{\left(x,y\right), \left(y,x\right)\right\}$.  
It is important to observe that 
$\left\{\left(x,y\right), \left(y,x\right)\right\}$ is the graph of a dual order automorphism $\beta: X \to X$ defined by $\beta(x) = y$ and $\beta(y) = x$. 
\end{ex}

In contrast to 
Example~\ref{ex:prime-for-6elmt},
consider the posets $\left(X, \le\right)$ and $\left(X^2, \preceq\right)$ in Figure \ref{fig:ex_construction_50elements} again. The $50$-element 
lattice $\mathsf{Up}((X^2,\preceq))$ 
has two atoms, 
$\left\{\left(x,y\right)\right\}$ and $\left\{\left(x,z\right)\right\}$, and two co-atoms, $\left(\left\{\left(x,y\right)\right\}^c\right)^\smile$ and $\left(\left\{\left(x,z\right)\right\}^c\right)^\smile$. 
Since 
$'$ must map atoms to co-atoms and vice versa, and the fact that there are only two possible candidates for $0$ (see Example~\ref{ex:Vposet}), it can be shown that it is impossible to define an involutive operation $'$ that will satisfy \textsf{(Dm)} and \textsf{(Dp)}. 

Note that the poset in Figure~\ref{fig:ex_construction_50elements} is not self-dual, whereas the one in Figure~\ref{fig:ex_construction_6elements} is. Example~\ref{ex:prime-for-6elmt} suggests that a dual order-automorphism could be used in the definition of a suitable $'$ operation. 

\begin{lem}
Let $\mathbf X = \left(X, \le\right)$ be a poset and $E$ an equivalence relation on $X$ such that ${\le} \subseteq E$. Let $\alpha: X \to X$ be an order automorphism of $\mathbf X$ and $\beta: X \to X$ a dual order automorphism of $\mathbf X$ such that $\alpha, \beta \subseteq E$. If $R \in \mathsf{Up}\left(\mathbf E\right)$, then $\alpha\mathbin{;} \beta\mathbin{;} R^c\mathbin{;} \beta \in \mathsf{Up}\left(\mathbf E\right)$ and $\beta\mathbin{;} R^c\mathbin{;} \beta\mathbin{;}\alpha \in \mathsf{Up}\left(\mathbf E\right)$.
\end{lem}

\begin{proof}
This follows from Lemma \ref{lem:important_up-_and_down-sets}. 
\end{proof}

We are now ready to show that we can give $\mathsf{Up}(\mathbf{E})$ the structure of a distributive quasi relation algebra. A dual order automorphism $\beta$ will be used in the definition of $'$. In order to guarantee that $'$ is involutive, we require $\beta$ to be self-inverse. To ensure that ${\sim}$,$-$, and $'$ interact in the correct way (i.e. that (\textsf{Di}) and (\textsf{Dp}) is satisfied) we require that $\beta = \alpha \mathbin{;} \beta \mathbin{;} \alpha$. 

\begin{thm}\label{thm:main_result}
Let $\mathbf{X}=\left(X,\le\right)$ be a poset and $E$ an equivalence relation on $X$ such that ${\le} \subseteq E$.  Let $\alpha: X \to X$ be an order automorphism of $\mathbf X$ and $\beta: X \to X$ a self-inverse dual order automorphism of $\mathbf X$ such that $\alpha, \beta \subseteq E$ and $\beta = \alpha \mathbin{;} \beta\mathbin{;}\alpha$. 
Set $1= {\leqslant}$ and $0 = \alpha \mathbin{;} (\le^c)^\smile$.
For $R \in \mathsf{Up}(\mathbf E)$, define  
$R' =  \alpha\mathbin{;} \beta \mathbin{;} R^c \mathbin{;} \beta$.
Then the algebra $\mathbf{Dq}(\mathbf E) = \left\langle \mathsf{Up}\left(\mathbf{E}\right),\cap, \cup, \mathbin{;}, 1, 0, {\sim}, {-}, '\right\rangle$ 
is a distributive quasi relation algebra. If $\alpha$ is the identity, then $\mathbf{Dq}(\mathbf E)$ is a cyclic distributive quasi relation algebra.	
\end{thm}	

\begin{proof}
Since $\left\langle \mathsf{Up}\left(\mathbf{E}\right),\cap, \cup, \mathbin{;}, 1, 0, {\sim}, {-}\right\rangle$ is an InFL-algebra (Theorem \ref{thm:InFL-algebra}), it will follow that $\mathbf{Dq}\left(\mathbf E\right)$ is a quasi relation algebra if we can show that $R'' = R$ for all $R \in  \mathsf{Up}\left(\mathbf{E}\right)$ and that 
\textsf{(Dm)}, \textsf{(Di)} and \textsf{(Dp)} hold. 
Let $R \in \mathsf{Up}(\mathbf{E})$. 
Then: 
\begin{align*}
R'' & = \alpha \mathbin{;} \beta\mathbin{;} \left(R'\right)^c\mathbin{;} \beta & \left(\textnormal{definition of } '\right)\\
& = \alpha \mathbin{;} \beta\mathbin{;} \left(\alpha \mathbin{;} \beta\mathbin{;} R^c \mathbin{;} \beta\right)^c\mathbin{;} \beta  & \left(\textnormal{definition of } '\right) \\
& = \alpha \mathbin{;} \beta\mathbin{;} \alpha \mathbin{;} \beta \mathbin{;} \left(R^c\right)^c \mathbin{;} \beta\mathbin{;} \beta  & \left(\textnormal{Lemma \ref{lem:important_eq_injective_map}}\right)\\
& = \beta \mathbin{;} \beta\mathbin{;} R \mathbin{;} \beta\mathbin{;} \beta & \left(\beta = \alpha\mathbin{;} \beta \mathbin{;} \alpha\right)\\ 
& = R  & \left(\beta = \beta^\smile\right)
\end{align*}	
	
To show that \textsf{(Dm)} holds, let $R, S \in  \mathsf{Up}\left(\mathbf{E}\right)$. Then: 
\begin{align*}
\left(R\cup S\right)' & = \alpha \mathbin{;} \beta \mathbin{;} \left(R \cup S\right)^c\mathbin{;} \beta & \left(\textnormal{definition of } '\right)\\
& = \alpha \mathbin{;} \beta\mathbin{;} \left(R^c \cap S^c\right)\mathbin{;} \beta \\	
& = \left(\alpha\mathbin{;} \beta \mathbin{;} R^c\mathbin{;} \beta\right) \cap \left(\alpha\mathbin{;} \beta\mathbin{;} S^c\mathbin{;} \beta\right) & \left(\textnormal{Lemma \ref{lem:important_eq_injective_map}}\right) \\	
& = R' \cap S' & \left(\textnormal{definition of } '\right)
\end{align*}	
	
For \textsf{(Di)}, let $R \in \mathsf{Up}\left(\mathbf E\right)$. Then we have
\begin{align*}
\left({\sim} R\right)' & = \alpha\mathbin{;} \beta \mathbin{;} \left({\sim}R\right)^c \mathbin{;} \beta & \left(\textnormal{definition of } '\right)\\
& = \alpha \mathbin{;} \beta \mathbin{;} \left(\left(R^c\right)^\smile \mathbin{;} \alpha \right)^c \mathbin{;} \beta & \left(\textnormal{Lemma \ref{lem:twiddle_minus_alternative}}\right)\\
& = \left(\alpha \mathbin{;} \beta \mathbin{;} \left(R^c\right)^\smile \mathbin{;} \alpha \mathbin{;} \beta\right)^c & \left(\textnormal{Lemma \ref{lem:important_eq_injective_map}}\right)\\
& = \left(\alpha \mathbin{;} \beta \mathbin{;} \left(\alpha^\smile \mathbin{;} R^c\right)^\smile \mathbin{;} \beta\right)^c & \left(\textnormal{Proposition \ref{prop:properties_binary_relations}}\right)\\
& = \left(\alpha \mathbin{;} \left(\beta^\smile \mathbin{;} \alpha^\smile \mathbin{;} R^c \mathbin{;} \beta^\smile\right)^\smile\right)^c& \left(\textnormal{Proposition \ref{prop:properties_binary_relations}}\right)\\
& = \left(\alpha \mathbin{;} \left(\beta \mathbin{;} \alpha^\smile \mathbin{;} R^c \mathbin{;} \beta\right)^\smile\right)^c & \left(\beta = \beta^\smile\right)\\
& = \alpha \mathbin{;} \left(\left(\beta \mathbin{;} \alpha^\smile \mathbin{;}  R^c \mathbin{;} \beta\right)^\smile\right)^c & \left(\textnormal{Lemma \ref{lem:important_eq_injective_map}}\right)\\
& = \alpha \mathbin{;} \left(\left(\alpha\mathbin{;} \beta \mathbin{;} R^c \mathbin{;} \beta\right)^\smile\right)^c & \left(
\beta\mathbin{;}\alpha^\smile = \alpha \mathbin{;} \beta\right)\\
& = \alpha \mathbin{;} \left(\left(R'\right)^c\right)^\smile & \left(\textnormal{definition of } '\right)\\
& = - \left(R'\right) & \left(\textnormal{Lemma \ref{lem:twiddle_minus_alternative}}\right)
\end{align*}	
	
For \textsf{(Dp)}, let $R, S \in \mathsf{Up}\left(\mathbf E\right)$. Then:
\begin{align*}
& R' + S' \\
& = {\sim}\left(-\left(S'\right)\mathbin{;} -\left(R'\right)\right)\\
& = {\sim}\left(-\left(\alpha \mathbin{;} \beta \mathbin{;} S^c\mathbin{;} \beta\right) \mathbin{;} - \left(\alpha \mathbin{;} \beta \mathbin{;} R^c \mathbin{;} \beta\right)\right) & \left(\textnormal{definition of }'\right)\\
& = {\sim} \left(\alpha \mathbin{;} \left(\left(\alpha \mathbin{;} \beta \mathbin{;} S^c\mathbin{;} \beta\right)^c\right)^\smile \mathbin{;} \alpha \mathbin{;} \left( \left(\alpha \mathbin{;} \beta \mathbin{;} R^c \mathbin{;} \beta\right)^c\right)^\smile \right) & \left(\textnormal{Lemma \ref{lem:twiddle_minus_alternative}}\right)\\
& = {\sim} \left(\alpha \mathbin{;} \left(\alpha \mathbin{;} \beta \mathbin{;} S\mathbin{;} \beta\right)^\smile \mathbin{;} \alpha \mathbin{;} \left(\alpha \mathbin{;} \beta \mathbin{;} R \mathbin{;} \beta\right)^\smile \right) & \left(\textnormal{Lemma \ref{lem:important_eq_injective_map}}\right)\\
& = {\sim} \left(\alpha \mathbin{;} \beta^\smile \mathbin{;}  S^\smile \mathbin{;} \beta^\smile \mathbin{;} \alpha^\smile \mathbin{;} \alpha \mathbin{ ;} \beta^\smile \mathbin{;} R^\smile \mathbin{;} \beta^\smile \mathbin{;} \alpha^\smile\right) & \left(\textnormal{Proposition \ref{prop:properties_binary_relations}}\right) \\
& = {\sim}\left(\alpha \mathbin{;} \beta \mathbin{;} S^\smile \mathbin{;}\beta \mathbin{;} \alpha^\smile \mathbin{;} \alpha \mathbin{;} \beta \mathbin{;} R^\smile \mathbin{;} \beta \mathbin{;} \alpha^\smile\right) & \left(\beta = \beta^\smile\right)\\
& = {\sim}\left(\alpha \mathbin{;} \beta \mathbin{;} S^\smile \mathbin{;} R^\smile \mathbin{;} \beta \mathbin{;} \alpha^\smile\right) & \left(\beta = \beta^\smile\right)\\
& = \left(\left(\alpha \mathbin{;} \beta \mathbin{;} S^\smile \mathbin{;} R^\smile \mathbin{;} \beta \mathbin{;} \alpha^\smile\right)^c\right)^\smile \mathbin{;} \alpha & \left(\textnormal{Lemma \ref{lem:twiddle_minus_alternative}}\right) \\
& = \left(\alpha \mathbin{;} \beta^\smile \mathbin{;} R \mathbin{;} S \mathbin{;} \beta^\smile \mathbin{;} \alpha^\smile\right)^c \mathbin{;} \alpha & \left(\textnormal{Proposition \ref{prop:properties_binary_relations}}\right)\\
& = \left(\alpha \mathbin{;} \beta \mathbin{;} R \mathbin{;} S \mathbin{;} \beta \mathbin{;} \alpha^\smile\right)^c \mathbin{;} \alpha & \left(\beta = \beta^\smile\right)\\
& = \alpha \mathbin{;} \beta \mathbin{;} \left(R \mathbin{;} S\right)^c \mathbin{;} \beta \mathbin{;} \alpha^\smile \mathbin{;} \alpha & \left(\textnormal{Lemma \ref{lem:important_eq_injective_map}}\right)\\
& = \alpha \mathbin{;} \beta \mathbin{;} \left(R \mathbin{;} S\right)^c \mathbin{;} \beta\\
& = \left(R \mathbin{;} S\right)' & \left(\textnormal{definition of }'\right)
\end{align*}
Finally, if $\alpha = \text{id}_X$, then by Theorem~\ref{thm:InFL-algebra} the DqRA will be cyclic. 
\end{proof}

Following the terminology for relation algebras~(cf. \cite[p.142]{Mad06}), we 
will refer to the algebra $\mathbf{Dq}(\mathbf E)$ as an \emph{equivalence distributive quasi relation algebra}. The class of equivalence distributive quasi relation algebras will be denoted by $\mathsf{EDqRA}$.

Since $X^2 = X \times X$ is an equivalence relation, we obtain: 

\begin{cor}\label{cor:FDqRA}
Let $\mathbf{X}=\left(X,\le\right)$ be a poset, and let 
$\alpha:X \to X$ be an order automorphism of $\mathbf X$, 
$\beta: X \to X$ a self-inverse dual order automorphism of $\mathbf X$,
and $\beta = \alpha \mathbin{;} \beta\mathbin{;} \alpha$. 
Set $1= {\leqslant}$ and $0 = \alpha \mathbin{;}(\le^c)^\smile = (\le^c)^\smile \mathbin{;} \alpha$. 
For $R \in \mathsf{Up}\left(\left(X^2, \preceq\right)\right)$, define  
$R' =  \alpha\mathbin{;} \beta \mathbin{;}  R^c \mathbin{;} \beta$.
Then the algebra $$\mathbf{Dq}\left(\left(X^2, \preceq\right)\right) = \left\langle \mathsf{Up}\left(\left(X^2, \preceq\right)\right),\cap, \cup, \mathbin{;}, 1, 0, {\sim}, {-}, '\right\rangle$$
is a distributive quasi relation algebra. 
\end{cor}

We will refer to this algebra as a \emph{full distributive quasi relation algebra} and denote the class of full distributive quasi relation algebras by $\mathsf{FDqRA}$. 

\begin{rem}\label{rem:aba=b}
To see that the condition $\alpha\mathbin{;} \beta \mathbin{;} \alpha = \beta$ is essential, consider the three-element discrete poset $X=\{x,y,z\}$ with $\alpha=\{(x,y),(y,x),(z,z)\}$ and $\beta=\{(x,y),(y,z),(z,x)\}$. 
Hence
$(x,z) \in \alpha \mathbin{;} \beta \mathbin{;} \alpha$ but $(x,z)\notin \beta$.~Now 
$\{(x,z)\} \in \mathsf{Up}((X^2,\preceq))$ and 
$({\sim}\{(x,z)\})'=\{(x,z)\}$ but $-\{(x,z)\}'=\{(y,x)\}$.
\end{rem}

\begin{rem}\label{rem:Rstar}
It would be natural to ask: 
is the algebra obtained by setting $R' =\beta\mathbin{;} R^c\mathbin{;} \beta\mathbin{;} \alpha$ also a distributive qRA? The answer is ``yes''. If we set $R' = \alpha\mathbin{;} \beta\mathbin{;} R^c\mathbin{;} \beta$ like in Theorem \ref{thm:main_result}, then 
\begin{align*}
R^{\triangledown 1} = {\sim}{\sim} R' & = \left[\left[\left(\left(\alpha\mathbin{;} \beta\mathbin{;} R^c\mathbin{;} \beta\right)^c\right)^\smile \mathbin{;} \alpha\right]^c\right]^\smile \mathbin{;} \alpha \\
& = \alpha^\smile \mathbin{;} \left(\alpha\mathbin{;} \beta \mathbin{;} R^c\mathbin{;} \beta\right)\mathbin{;} \alpha\\
& = \beta \mathbin{;} R^c\mathbin{;} \beta \mathbin{;} \alpha
\end{align*}
and the resulting algebra is a  DqRA according to Theorem~\ref{thm:star-qRA}. 
\end{rem} 

Now we will show that the DqRAs generated in this way (i.e. by letting $R'=\beta \mathbin{;} R^c \mathbin{;}  \beta \mathbin{;} \alpha$) are all in $\mathsf{EDqRA}$. First, we need 
a technical lemma. 
The lemma below is easily proved, with Lemma~\ref{lem:comp_and_conv_in_E} needed for (ii).

\begin{lem}\label{lem:properties_of_alpha^nbeta_and_betaalpha^n}
Let $\mathbf{X}=\left(X,\le\right)$ be a poset and $E$ an equivalence relation on $X$ such that ${\le} \subseteq E$.  Let $\alpha:X \to X$ be an order automorphism of $\mathbf X$ and $\beta: X \to X$ a self-inverse dual order automorphism of $\mathbf X$ such that $\alpha, \beta \subseteq E$ and $\beta = \alpha \mathbin{;} \beta\mathbin{;} \alpha$. Then the following hold for all $n \in \omega$:
\begin{enumerate}[\normalfont (i)]
\item $\alpha^n\mathbin{;} \beta$ and $\beta\mathbin{;} \alpha^n$ are dual order automorphisms of $\mathbf X$;
\item $\alpha^n\mathbin{;} \beta \subseteq E$ and $\beta\mathbin{;} \alpha^n \subseteq E$; 
\item $\alpha^n\mathbin{;} \beta$ and $\beta\mathbin{;} \alpha^n$ are self-inverse;
\item $\alpha \mathbin{;} \left( \alpha^{n}\mathbin{;} \beta\right)\mathbin{;} \alpha = \alpha^n\mathbin{;} \beta$ and $\alpha\mathbin{;} \left( \beta \mathbin{;} \alpha^{n}\right) \mathbin{;} \alpha = \beta\mathbin{;} \alpha^n$.
\end{enumerate}
\end{lem}

Lemmas~\ref{lem:general_twiddle_and _minus} and~\ref{lem:properties_of_alpha^nbeta_and_betaalpha^n} will now be 
used to provide an alternative expression for $R^{\triangledown n}$ in terms of compositions of $\alpha$ and $\beta$. 

\begin{prop}\label{prop:star_qRA-in_EDqRA}
Let $\mathbf{A}=\left\langle 
\mathsf{Up}(\mathbf{E}), \cap, \cup, \mathbin{;}, 1,0, {\sim},{-}, '\right\rangle$ be the DqRA given by Theorem~\ref{thm:main_result}. 
Then there are self-inverse dual order automorphisms $\beta_{\triangledown n}$ and $\beta_{\vartriangle n}$ of $\mathbf X$ such that
\begin{enumerate}[\normalfont (i)] 
\item $\beta_{\triangledown n}, \beta_{\vartriangle n} \subseteq E$;
\item $\beta_{\triangledown n} = \alpha\mathbin{;} \beta_{\triangledown n}\mathbin{;} \alpha$ and $\alpha\mathbin{;} \beta_{\vartriangle n}\mathbin{;} \alpha = \beta_{\vartriangle n}$;
\item $R^{\triangledown n} = \alpha\mathbin{;} \beta_{\triangledown n}\mathbin{;} R^c\mathbin{;} \alpha$ and 
$R^{\vartriangle n} = \alpha\mathbin{;} \beta_{\vartriangle n}\mathbin{;} R^c\mathbin{;} \alpha$.
\end{enumerate} 
\end{prop}	

\begin{proof}
We will prove the existence of a 
$\beta_{\vartriangle n}$ satisfying the conditions (i)--(iii) and leave the proof of the existence of a 
$\beta_{\triangledown n}$ satisfying (i)--(iii) for the reader. We have 
\begin{align*}
R^{\vartriangle n} = {-}^{2n}R'
& = \alpha^n \mathbin{;} R' \mathbin{;} \left(\alpha^{\smile}\right)^n & (\text{Lemma \ref{lem:general_twiddle_and _minus}})\\
& = \alpha^n \mathbin{;} \alpha\mathbin{;} \beta \mathbin{;} R^c\mathbin{;} \beta \mathbin{;} \left(\alpha^{\smile}\right)^n\\
& = \alpha^{n+1} \mathbin{;} \beta \mathbin{;} R^c\mathbin{;} \beta^\smile \mathbin{;}  \left(\alpha^{n}\right)^\smile & \left(\beta^\smile = \beta\right)\\
& =\alpha^{n+1} \mathbin{;} \beta \mathbin{;} R^c\mathbin{;} \left(\alpha^n\mathbin{;} \beta\right)^\smile & \left(\text{Proposition \ref{prop:properties_binary_relations}}\right)\\
& = \alpha\mathbin{;} \alpha^{n} \mathbin{;} \beta \mathbin{;} R^c\mathbin{;} \alpha^n\mathbin{;} \beta  & 
\end{align*}
Hence, if we set 
$\beta_{\vartriangle n} = \alpha^n\mathbin{;} \beta$, the result follows from Lemma \ref{lem:properties_of_alpha^nbeta_and_betaalpha^n}. 
\end{proof}

The above proposition shows that for all $n \in \omega$, the concrete algebras
$\left\langle \mathsf{Up}\left(\mathbf{E}\right),\cap, \cup, \mathbin{;}, 1,0, {\sim}, {-}, ^{\triangledown n}\right\rangle$  and 
$\left\langle \mathsf{Up}\left(\mathbf{E}\right),\cap, \cup, \mathbin{;},1 , 0, {\sim}, {-}, ^{\vartriangle n}\right\rangle$ are in the class $\mathsf{EDqRA}$. In particular, this shows that $\left\langle \mathsf{Up}\left(\mathbf{E}\right),\cap, \cup, \mathbin{;}, 1, 0, {\sim}, {-}, ^{\triangledown 1}\right\rangle$ is in \textsf{EDqRA}. That is, all algebras constructed with $R'=\beta \mathbin{;} R^c \mathbin{;} \beta \mathbin{;} \alpha$ are in \textsf{EDqRA}. This demonstrates that the definition does not depend on whether $\alpha$ is used on the left or the right in the definition of $R'$.

\section{Representable distributive quasi relation algebras}\label{sec:RDqRA}

In this section, we will first show that $\mathbb{ISP}\left(\mathsf{FDqRA}\right) =  \mathbb{IS}\left(\mathsf{EDqRA}\right)$. Our strategy mimics the analogous results for relation algebras (cf.~\cite[Chapter 3]{Mad06}). 

Consider an index set $I$. For each index $i \in I$, let $\mathbf X_i  = \left(X_i, \le_i\right)$ be a poset and $E_i$ an equivalence relation such that ${\le}_i \subseteq E_i$. We further require that $X_i \cap X_j = \varnothing$ whenever $i \neq j$. Set 
$$
X := \displaystyle \bigcup_{i\in I}X_i, \quad \displaystyle{\le} :=\bigcup_{i\in I}{\le}_i \quad \textnormal{and} \quad E := \displaystyle \bigcup_{i \in I}E_i.
$$
Given maps $\gamma_i: X_i\rightarrow X_i$, $i \in I$, define $\gamma: X \to X$ by setting, for each $x \in X$,
$\gamma(x) = \gamma_i(x) \textnormal{ if } x \in X_i$.
It is not difficult to see that $\gamma = \displaystyle\bigcup_{i\in I}\gamma_i$. 
These definitions will be used throughout the theorems and their proofs in this section.

\begin{thm}\label{thm:4.1} 
Let  $I$ be an index set. For each $i \in I$, let $\mathbf X_i  = \left(X_i, \le_i\right)$ be a poset and $E_i$ an equivalence relation such that ${\le}_i\; \subseteq E_i$. Assume  $X_i \cap X_j = \varnothing$ for all $i, j \in I$ such that $i \neq j$. Suppose further that for each $i \in I$ the map \break $\alpha_i: X_i \rightarrow X_i$ is an order automorphism of $\mathbf X_i$ and $\beta_i: X_i \rightarrow X_i$ is a self-inverse dual order automorphism of  $\mathbf X_i$ such that $\alpha_i, \beta_i \subseteq E_i$ and $\beta_i = \alpha_i \mathbin{;} \beta_i \mathbin{;} \alpha_i$. Then:
\begin{enumerate}[\normalfont (i)]
\item $\left(X, \le\right)$ is a poset and $E$ an equivalence relation on $X$ such that ${\le} \subseteq E$;
\item $\alpha$ is an order automorphism of $\mathbf X$ such that $\alpha \subseteq E$;
\item $\beta$ is a self-inverse dual order automorphism of $\mathbf X$ such that $\beta \subseteq E$;
\item $\beta = \alpha \mathbin{;} \beta \mathbin{;}\alpha$;
\item $\mathbf{Dq}\left(\mathbf E\right) \cong \displaystyle \prod_{i \in I}\mathbf{Dq}\left(\mathbf E_i\right)$. 
\end{enumerate}
\end{thm}

\begin{proof}
Item (i) is straightforward. We will prove item (iii) and omit the proof of (ii) since it is similar. Firstly, $\beta \subseteq E$ since $\beta_i \subseteq E_i$ for all $i \in I$. To see that $\beta$ is a dual order automorphism, let $x, y \in X$ and assume $x \le y$. Then $x, y \in X_j$ and $x \le_j y$ for some $j \in I$. Since $\beta_j$ is a dual order automorphism, $x \le_j y$ implies $\beta_j(y) \le_j \beta_j(x)$. But $x, y \in X_j$, so $\beta(x) = \beta_j(x) \in X_j$ and $\beta(y) = \beta_j(y) \in X_j$, and therefore $\beta(y) \le \beta(x)$. 

For the  converse implication, suppose $x \not\le y$. Then there is some $j \in I$ such that $x, y \in X_j$ but $x \not\le_j y$. Applying the fact that $\beta_j$ is a dual order automorphism to the latter gives $\beta_j(y) \not\le \beta_j(x)$. Hence, since $x, y \in X_j$ implies that $\beta(x) = \beta_j(x) \in X_j$ and $\beta(y) = \beta_j(y) \in X_j$, we get $\beta(y) \not\le \beta(x)$. 

It remains to show that $\beta$ is self-inverse. To this end, let $x \in X$. Then $x \in X_j$ for some $j \in I$. Hence, $\beta(x) = \beta_j(x)$, and so, since $\beta_j(x) \in X_j$, we get $\beta(\beta(x)) = \beta(\beta_j(x)) = \beta_j(\beta_j(x))$. But $\beta_j$ is self-inverse, so $\beta(\beta(x)) = \beta_j(\beta_j(x)) = x$.

For item (iv), let $x \in X$. Then $x \in X_j$ for some $j \in I$. We thus get $\alpha(x) = \alpha_j(x) \in X_j$. Hence, $\beta\left(\alpha(x)\right) = \beta_j\left(\alpha_j\left(x\right)\right)$, and therefore, since $\beta_j\left(\alpha_j\left(x\right)\right) \in X_j$, we have $\alpha\left(\beta\left(\alpha\left(x\right)\right)\right) = \alpha_j\left(\beta_j\left(\alpha_j\left(x\right)\right)\right)$. But $\alpha_j\mathbin{;}\beta_j\mathbin{;} \alpha_j = \beta_j$, so $\alpha\left(\beta\left(\alpha\left(x\right)\right)\right) = \beta_j(x) = \beta(x)$. 

To prove (v), define a map $\varphi$: $\mathsf{Up}\left(\mathbf E\right) \rightarrow 
\prod_{j \in I}\mathsf{Up}\left(\mathbf E_j\right)$ coordinatewise by setting, for each $R \in \mathsf{Up}\left(\mathbf E\right)$ and each $i \in I$,
$\varphi(R)(i) := R \cap E_i$.
We now show that $\varphi$ is an isomorphism from $\mathbf{Dq}\left(\mathbf E\right)$ into $\prod_{j \in I}\mathbf{Dq}\left(\mathbf E_j\right)$.
Showing that $\varphi(R \cap S)(i)= \left(\varphi(R) \cap \varphi(S) \right) (i)$ and 
$\varphi(R \cup S)(i)= \left(\varphi(R) \cup \varphi(S) \right) (i)$ follows from basic set-theoretic calculations. 

Now we show that $\varphi$ preserves the monoid identity. 
In the second last step below we use the fact that 
${\le}_j \cap E_i = \varnothing$ for $j\neq i$, and in the last step we use ${\le}_j \subseteq E_i$: 
$$
\varphi(\le)(i)  = {\le} \cap E_i 
= \bigcup_{j\in I}{\le}_j \cap E_i
= \bigcup_{j\in I}\left({\le}_j \cap E_i\right)
= {\le}_i \cap E_i  
= {\le}_i.$$

Below we will use the $-$ symbol to denote set complement (not
to be confused with the unary operation on a qRA). By Lemma~\ref{lem:important_eq_injective_map} we get:
$$ \varphi(0)(i) = \varphi\left(\alpha \mathbin{;} \left(\le^c\right)^\smile\right)(i) 
 = \left(\alpha \mathbin{;} \left(\le^\smile\right)^c\right) \cap E_i
 = \left(\alpha\mathbin{;} {\le}^\smile\right)^c \cap E_i.  $$
Using the above we calculate 
\begin{align*}
\varphi(0)(i) 
& = \left[ E - (\alpha \mathbin{;} {\leqslant}^\smile)\right] \cap E_i \\ 
& = E_i - \left[ (\alpha \mathbin{;} {\leqslant}^\smile) \cap E_i \right] & \left(E_i\subseteq E\right) \\
& = E_i - \left[ (\alpha \mathbin{;} {\leqslant}^\smile) \cap E_i^\smile \right] & \left(E_i = E_i^\smile\right) \\
& = E_i - \left(\left({\le} \mathbin{;} \alpha^\smile\right)^\smile \cap E_i^\smile\right) & \left(\textnormal{\textnormal{Proposition \ref{prop:properties_binary_relations}}}\right)\\
& =E_i - \left[\left(\left({\le} \mathbin{;}  \alpha^\smile\right) \cap E_i\right)^\smile\right] & \left(\textnormal{\textnormal{Proposition \ref{prop:properties_binary_relations}}}\right)
\end{align*}
Now by the definition of $\alpha$ we get: 
\begin{align*}
& \varphi(0)(i)\\& = E_i - \left[\left(\left(\bigcup_{j\in I}{\le}_j \mathbin{;}\left(\bigcup_{k \in I}\alpha_k\right)^\smile\right) \cap E_i\right)^\smile\right]\\
& = E_i- \left[\left(\left(\bigcup_{j \in I}{\le}_j\mathbin{;} \bigcup_{k\in I}\alpha_k^\smile\right) \cap E_i\right)^\smile\right] & \left(\textnormal{Proposition \ref{prop:properties_binary_relations}}\right)\\
& = E_i - \left[\left(\bigcup_{j \in I}\bigcup_{k\in I}\left({\le}_j\mathbin{;} \alpha_k^\smile\right) \cap E_i\right)^\smile\right] & \left(\textnormal{Proposition \ref{prop:properties_binary_relations}}\right)\\
& = E_i -\left[\left(\bigcup_{j \in I}\left({\le}_j\mathbin{;} \alpha^\smile_j\right) \cap E_i\right)^\smile\right] & \left({\le}_j \mathbin{;} \alpha_k^\smile = \varnothing \textnormal{ for } j\neq k\right)\\
& = E_i - \left[\left(\bigcup_{j \in I}\left(\left({\le}_j\mathbin{;} \alpha_j^\smile\right) \cap E_i\right)\right)^\smile\right]
\end{align*}
Using  
the fact that 
$\left({\le}_j \mathbin{;} \alpha_j^\smile\right) \cap E_i = \varnothing \textnormal{ for } j\neq i$, the set above is equal to: 
\begin{align*}
\left[\left(\left({\le}_i\mathbin{;} \alpha_i^\smile\right) \cap E_i\right)^\smile\right]^c  
& = \left[\left({\le}_i\mathbin{;} \alpha_i^\smile\right)^\smile\right]^c & \left({\le}_i \mathbin{;} \alpha_i^\smile \subseteq E_i\right)\\
& = \left(\alpha_i\mathbin{;}{\le}_i^\smile\right)^c & \left(\textnormal{Proposition \ref{prop:properties_binary_relations}}\right)\\
& = \alpha_i\mathbin{;} \left({\le}_i^\smile\right)^c &
\left(\textnormal{Lemma \ref{lem:important_eq_injective_map}}\right)\\
& = 0_i.
\end{align*}

We will now demonstrate that $\varphi$ preserves the $'$ operation. 
\begin{align*}
\varphi(R')(i)	
 &= \varphi\left(\alpha \mathbin{;} \beta \mathbin{;} R^c \mathbin{;} \beta\right)(i)\\ &= \left(\alpha \mathbin{;} \beta \mathbin{;} R^c \mathbin{;} \beta\right) \cap E_i \\
& = \left(\alpha \mathbin{;} \beta \mathbin{;} R \mathbin{;} \beta\right)^c \cap E_i &
\left(\textnormal{Lemma \ref{lem:important_eq_injective_map}}\right)\\
& = E_i - \left(\left(\alpha \mathbin{;} \beta \mathbin{;} R \mathbin{;} \beta\right) \cap E_i\right).
\end{align*}
Applying the definition of $\alpha$ and $\beta$ gives us: 
$$\varphi(R')(i) =  E_i - \left(\left(\bigcup_{j\in I}\alpha_j \mathbin{;} \bigcup_{k\in I}\beta_k \mathbin{;} R \mathbin{;} \bigcup_{\ell\in I}\beta_\ell\right) \cap E_i\right).$$
We can then use part (viii) and (ix) from Proposition~\ref{prop:properties_binary_relations}, as well as the fact that $\alpha_j \mathbin{;} \beta_k = \varnothing$ for $j \neq k$, and finally $\alpha_j \mathbin{;} \beta_j \mathbin{;} R \mathbin{;} \beta_{\ell} = \varnothing$ when $j \neq \ell$ to get:  

\begin{align*}
\varphi(R')(i) &= 
E_i - \left(\bigcup_{j\in I}\left(\alpha_j\mathbin{;} \beta_j \mathbin{;} R\mathbin{;} \beta_j\right) \cap E_i\right) \\
 &=  E_i - \left(\bigcup_{j\in I}\left(\left(\alpha_j\mathbin{;} \beta_j \mathbin{;} R\mathbin{;} \beta_j\right) \cap E_i\right)\right).
\end{align*}
Since $\left(\alpha_j\mathbin{;} \beta_j \mathbin{;} R\mathbin{;}\beta_j\right) \cap E_i = \varnothing \textnormal{ for } j \neq i$, the set above is equal to  
\begin{align*} 
\left(\left(\alpha_i\mathbin{;} \beta_i \mathbin{;} R\mathbin{;} \beta_i\right) \cap E_i\right)^c & =  \left(\alpha_i\mathbin{;} \beta_i \mathbin{;} \left(R \cap E_i\right) \mathbin{;} \beta_i\right)^c  & \left(\textnormal{Lemma } \ref{lem:(gammaR)capE=gamma(RcapE)}\right) \\
& = \left(\alpha_i\mathbin{;} \beta_i \mathbin{;} \varphi(R)(i) \mathbin{;} \beta_i\right)^c \\
& = \alpha_i\mathbin{;} \beta_i \mathbin{;} \left(\varphi(R)(i)\right)^c \mathbin{;} \beta_i &
\left(\textnormal{Lemma \ref{lem:important_eq_injective_map}}\right)\\
& = \left(\varphi(R)(i)\right)'. && \qedhere
\end{align*}
\end{proof}

Before we proceed to our next theorem, we need a technical lemma. 

\begin{lem}\label{lem:gamma_and_E}
Let $X$ be a set and $E$ an equivalence relation on $X$. If $\gamma: X \rightarrow X$ is a function and $\gamma \subseteq E$, then 
$\gamma(y) \in \left[x\right]$ whenever $y \in \left[x\right]$. 
\end{lem}

\begin{proof}
Let $y \in \left[x\right]$. Then we have $(x, y) \in E$. But we know that $\gamma \subseteq E$, so $(y,\gamma(y)) \in E$. Hence, by the transitivity of $E$, $(x, \gamma(y)) \in E$, and therefore $\gamma(y) \in \left[x\right]$, as required. 
\end{proof}

\begin{thm}\label{thm:EDqRA_to_FDqRa}
Let $\mathbf{X}=\left(X,\le\right)$ be a poset and $E$ an equivalence relation on $X$ such that ${\le} \subseteq E$.  Let $\alpha: X \to X$ be an order automorphism of $\mathbf X$ and $\beta: X \to X$ a self-inverse dual order automorphism of $\mathbf X$ such that $\alpha, \beta \subseteq E$ and $\beta = \alpha \mathbin{;} \beta\mathbin{;} \alpha$. For each $x \in X$, define $\alpha_x: \left[x\right] \to \left[x\right]$ by $\alpha_x(y) = \alpha(y)$ and $\beta_x: \left[x\right] \to \left[x\right]$ by $\beta_x(y) = \beta(y)$ for all $y \in [x]$. 

\begin{enumerate}[\normalfont (i)]
\item For each $x \in X$, $\left(\left[x\right], {\le}\cap [x]^2\right)$ is a poset. 
\item For each $x \in X$, the map $\alpha_x$ is an order automorphism of $\left(\left[x\right], {\le}\cap [x]^2\right)$ such that $\alpha_x\subseteq \left[ x \right]^2 $. 
\item For each $x \in X$, the map $\beta_x$ is a self-inverse dual order automorphism of $\left(\left[x\right], {\le} \cap [x]^2
\right)$ such that $\beta_x\subseteq \left[x\right]^2$.
\item For each $x \in X$, we have $\alpha_x\mathbin{;} \beta_x\mathbin{;} \alpha_x = \beta_x$.
\item $\mathbf{Dq}\left(\mathbf E\right) \cong \displaystyle \prod_{x \in X}\mathbf{Dq}\left( \left(\left[x\right]^2, {\preceq}\cap \left([x]^2\right)^2\right)\right)$. 
\end{enumerate}
\end{thm}

\begin{proof}
Item (i) is straightforward. Items (ii) and (iii) are similar, so we will prove (iii) and leave (ii) for the reader. 

Firstly, $\beta_x \subseteq \left[x\right]^2$ since $\beta: \left[x\right] \to \left[x\right]$. To see that $\beta_x$ is a dual order automorphism, let $y, z \in \left[x\right]$ and assume 
$(y,z) \in {\le} \cap [x]^2$. 
Then $y \le z$, and so since $\beta$ is a dual order automorphism of $\mathbf X$, we have $\beta(z) \le \beta(y)$. But $\beta(z), \beta(y) \in \left[x\right]$ by Lemma \ref{lem:gamma_and_E}, and hence, since  $\beta(z) = \beta_x(z)$ and $\beta(y) = \beta_x(y)$, we get $\left(\beta_x(z),\beta_x(y)\right) \in {\le}\cap [x]^2$. 

Conversely, suppose $y, z\in \left[x\right]$ such that 
$(y,z) \notin  {\le} \cap [x]^2$. 
Then we have $y \not\le z$. Applying the fact that $\beta$ is a dual order automorphism to this gives $\beta(z) \not\le \beta(y)$. Hence, as $y, z \in \left[x\right]$ implies that $\beta(y), \beta(z) \in \left[x\right]$ (by Lemma~\ref{lem:gamma_and_E}), we have 
$\beta_x(z) = \beta(z)$  
and 
$\beta(y) = \beta_x(y)$ and so $\left(\beta_x(z), \beta_x(y)\right) \notin {\le}\cap[x]^2$.  

To see that $\beta_x$ is self-inverse, let $y \in \left[x\right]$. Then $\beta(y) \in \left[x\right]$ by Lemma~\ref{lem:gamma_and_E}. Hence, $\beta_x(\beta_x(y)) = \beta_x(\beta(y)) = \beta(\beta(y)) = y$ (since $\beta$ is self-inverse).  

For item (iv), let $y \in [x]$. Then $\alpha(y) \in \left[x\right]$ and $\beta\left(\alpha(y)\right) \in \left[x\right]$ by Lemma \ref{lem:gamma_and_E}. Hence, $\left(\alpha_x\circ\beta_x\circ\alpha_x\right)\left(y\right) = \alpha_x\left(\beta_x\left(\alpha_x(y)\right)\right) = \alpha_x\left(\beta_x\left(\alpha(y)\right)\right) = \alpha_x\left(\beta\left(\alpha_x(y)\right)\right) = \alpha\left(\beta\left(\alpha(y)\right)\right) = \beta(y) = \beta_x(y)$. 

Item (v) follows from Theorem~\ref{thm:4.1} with $E=\bigcup \{\, [x] \mid x\in X\,\}$. 
\end{proof}

We are now able to prove the main theorem of this section. 
\begin{thm}
$\mathbb{IP}\left(\mathsf{FDqRA}\right) = \mathbb{I}(\mathsf{EDqRA})$
\end{thm}

\begin{proof}
Let $\mathbf{X}=\left(X,\le\right)$ be a poset and $E$ an equivalence relation on $X$ such that ${\le} \subseteq E$. Further, let $\alpha: X \to X$ be an order automorphism of $\mathbf X$ and $\beta: X \to X$ a self-inverse dual order automorphism of $\mathbf X$ such that $\alpha, \beta \subseteq E$ and $\beta = \alpha \mathbin{;} \beta\mathbin{;} \alpha$. If $X \neq \varnothing$, then we can use Theorem \ref{thm:EDqRA_to_FDqRa} to conclude that $\mathbf{Dq}\left(\mathbf E\right) \in \mathbb{P}\left(\mathsf{FDqRA}\right)$. If $X = \varnothing$, then $\mathbf{Dq}\left(\left(E, \preceq\right)\right) = \mathbf{Dq}\left(\left(X^2, \preceq\right)\right)$, so $\mathbf{Dq}\left(\left(E, \preceq\right)\right) \in \mathbb{P}\left(\mathsf{FDqRA}\right)$. 

We will now prove the reverse inclusion. Let $I$ be an index set. For each $i \in I$, let $\mathbf X_i  = \left(X_i, \le_i\right)$ be a poset. Suppose further that for each $i \in I$ the map $\alpha_i: X_i \rightarrow X_i$ is an order automorphism of $\mathbf X_i$ and $\beta_i: X_i \rightarrow X_i$ is a self-inverse dual order automorphism of  $\mathbf X_i$ such that $\beta_i = \alpha_i \mathbin{;} \beta_i \mathbin{;} \alpha_i$. Let $\mathbf A = \prod_{i \in I}\mathbf{Dq}\left(\left(X^2_i, \preceq_i\right)\right)$. We may assume that $X_i\neq \varnothing$ for all $i \in I$ since the direct product of a family of algebras from $\mathsf{FDqRA}$ is isomorphic to another such product from which
all factors of the form $\mathbf{Dq}\left(\left(\varnothing, \varnothing\right)\right)$ have been deleted. Now, for each $i \in I$, set $X'_i = \left\{\,(x, i) \mid x \in X_i\,\right\}$. Then $X'_i \cap X'_j = \varnothing$ for $i \neq j$. For all $i \in I$, define $\le'_{i}$ on $X'_i$ by $(x, i)\le'_{i}(y, i)$ iff $x \le_i y$. It is not difficult to see that $(X'_i, \le'_i)$ is a poset for all $i \in I$. Further, for all $i \in I$, let $E_i' = \left\{\,\left(\left(x, i\right), \left(y, i\right)\right)\mid \left(x, y\right) \in E_i\,\right\}$. 
Clearly then $E'_i$ is an equivalence relation on $X'_i$ and ${\le}'_i \subseteq E'_i$ for all $i \in I$. Finally, for all $i \in I$ and $(x, i) \in X'_i$, define $\alpha'_i: X'_i \to X'_i$ and $\beta_i': X'_i \to X'_i$ by $\alpha'_i(x,i) = (\alpha_i(x), i)$ and $\beta'_i(x, i) = (\beta_i(x), i)$. It is straightforward to prove that $\alpha'_i$ is an automorphism and that $\beta'_i$ is a dual automorphism. Moreover, $\alpha'_i, \beta'_i \subseteq E'_i$ and 
$\alpha'_i\mathbin{;} \beta'_i \mathbin{;} \alpha_i'= \beta'_i$. Now set
\[
X' = \bigcup_{i\in I}X'_i, \quad \displaystyle{\le}'  =\bigcup_{i\in I}{\le}'_i, \quad E' = \displaystyle \bigcup_{i \in I}E'_i, \quad \alpha' = \bigcup_{i \in I}\alpha'_i  \quad \textnormal{ and } \quad \beta' = \bigcup_{i\in I}\beta_i'.
\]
Then $\prod_{i \in I}\mathbf{Dq}\left(\left(X^2_i, \preceq_i\right)\right) \cong \prod_{i \in I}\mathbf{Dq}\left(\left(X'^2_i, \preceq'_i\right)\right)$. Hence, it follows from Theorem \ref{thm:4.1} that $\mathbf{A} \cong \mathbf{Dq}\left(\mathbf{E}'\right)$.
\end{proof}

We are now ready to define the class of representable distributive quasi-relation algebras.  

\begin{df}\label{def:RDqRA}
A distributive quasi relation algebra 
$\mathbf{A}$ 
is \emph{representable} if 
$\mathbf{A} \in \mathbb{ISP}\left(\mathsf{FDqRA}\right)$
or, equivalently, $\mathbf{A} \in \mathbb{IS}\left(\mathsf{EDqRA}\right)$. 
\end{df}

Following the tradition for relation algebras, we will say that a DqRA is \emph{finitely representable} if it is representable using a finite poset. An immediate question is whether the class of finite DqRAs can be partitioned into those that are non-representable, those that are representable over a finite poset, and those that are representable but not over a finite poset (see the discussion of the relation algebra case by Maddux~\cite[Section 9]{Mad94}). 

\section{Representations of small distributive quasi relation algebras} \label{sec:small-rep}

In this section we give a number of examples of small 
DqRAs that are representable and also give examples for which we do not have representations but where we know the representation must be infinite. 
An algebra of particular interest is the
DqRA that has the 
three-element Sugihara chain as a reduct (see Figure~\ref{fig:representable_chains}(iii)). 
Jipsen and \v{Semrl}~\cite[Section 6]{JS23} observe that it is not representable 
by weakening relations in their set-up.
However, we are able to obtain a representation using a non-identity $\alpha$ (see Example~\ref{ex:chains1} below).

Below we give brief descriptions of representations for 15 different finite DqRAs. In the diagrams, black nodes denote idempotent elements. 
If a non-obvious product is not indicated, then it is $\bot$.
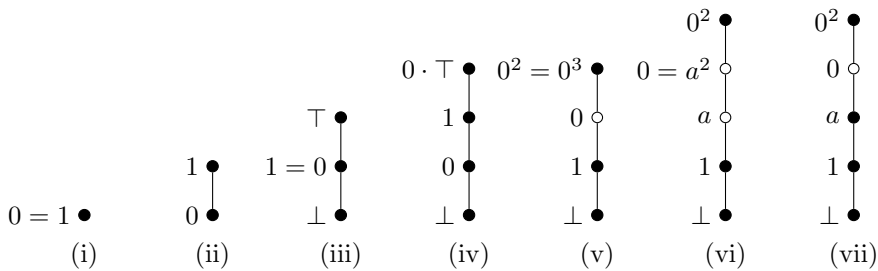
\begin{figure}[ht]
\centering
\begin{tikzpicture}[scale=0.65]
\begin{scope}[xshift=0.4cm]
\node[draw,circle,inner sep=1.5pt,fill] (0) at (0,0) {};
\node[label,anchor=east,xshift=1pt] at (0) {$0=1$};
\node[label,anchor=north,yshift=-4pt] at (0) {(i)};
\end{scope}
\begin{scope}[xshift=2.6cm]
\node[draw,circle,inner sep=1.5pt,fill] (bot) at (0,0) {};
\node[draw,circle,inner sep=1.5pt,fill] (top) at (0,1) {};
\draw[order] (bot)--(top);
\node[label,anchor=east,xshift=1pt] at (top) {$1$};
\node[label,anchor=east,xshift=1pt] at (bot) {$0$};
\node[label,anchor=north,yshift=-4pt] at (bot) {(ii)};
\end{scope}
\begin{scope}[xshift=5.2cm]
\node[draw,circle,inner sep=1.5pt,fill] (bot) at (0,0) {};
\node[draw,circle,inner sep=1.5pt,fill] (id) at (0,1) {};
\node[draw,circle,inner sep=1.5pt,fill] (top) at (0,2) {};
\draw[order] (bot)--(id)--(top);
\node[label,anchor=east,xshift=1pt] at (top) {$\top$};
\node[label,anchor=east,xshift=1pt] at (id) {$1=0$};
\node[label,anchor=east,xshift=1pt] at (bot) {$\bot$};
\node[label,anchor=north,yshift=-4pt] at (bot) {(iii)};
\end{scope}
\begin{scope}[xshift=7.8cm]
\node[draw,circle,inner sep=1.5pt,fill] (bot) at (0,0) {};
\node[draw,circle,inner sep=1.5pt,fill] (0) at (0,1) {};
\node[draw,circle,inner sep=1.5pt,fill] (1) at (0,2) {};
\node[draw,circle,inner sep=1.5pt,fill] (top) at (0,3) {};
\draw[order] (bot)--(0)--(1)--(top);
\node[label,anchor=east,xshift=1pt] at (top) {$0\cdot \top$};
\node[label,anchor=east,xshift=1pt] at (1) {$1$};
\node[label,anchor=east,xshift=1pt] at (0) {$0$};
\node[label,anchor=east,xshift=1pt] at (bot) {$\bot$};
\node[label,anchor=north,yshift=-4pt] at (bot) {(iv)};
\end{scope}
\begin{scope}[xshift=10.4cm]
\node[draw,circle,inner sep=1.5pt,fill] (bot) at (0,0) {};
\node[draw,circle,inner sep=1.5pt,fill] (1) at (0,1) {};
\node[unshaded] (0) at (0,2) {};
\node[draw,circle,inner sep=1.5pt,fill] (top) at (0,3) {};
\draw[order] (bot)--(1)--(0)--(top);
\node[label,anchor=east,xshift=1pt] at (top) {$0^2=0^3$};
\node[label,anchor=east,xshift=1pt] at (1) {$1$};
\node[label,anchor=east,xshift=1pt] at (0) {$0$};
\node[label,anchor=east,xshift=1pt] at (bot) {$\bot$};
\node[label,anchor=north,yshift=-4pt] at (bot) {(v)};
\end{scope}
\begin{scope}[xshift=13cm]
\node[draw,circle,inner sep=1.5pt,fill] (bot) at (0,0) {};
\node[draw,circle,inner sep=1.5pt,fill] (1) at (0,1) {};
\node[unshaded] (a) at (0,2) {};
\node[unshaded] (0) at (0,3) {};
\node[draw,circle,inner sep=1.5pt,fill] (top) at (0,4) {};
\draw[order] (bot)--(1)--(a)--(0)--(top);
\node[label,anchor=east,xshift=1pt] at (top) {$0^2$};
\node[label,anchor=east,xshift=1pt] at (0) {$0=a^2$};
\node[label,anchor=east,xshift=1pt] at (a) {$a$};
\node[label,anchor=east,xshift=1pt] at (1) {$1$};
\node[label,anchor=east,xshift=1pt] at (bot) {$\bot$};
\node[label,anchor=north,yshift=-4pt] at (bot) {(vi)};
\end{scope}
\begin{scope}[xshift=15.6cm]
\node[draw,circle,inner sep=1.5pt,fill] (bot) at (0,0) {};
\node[draw,circle,inner sep=1.5pt,fill] (1) at (0,1) {};
\node[draw,circle,inner sep=1.5pt,fill] (a) at (0,2) {};
\node[unshaded] (0) at (0,3) {};
\node[draw,circle,inner sep=1.5pt,fill] (top) at (0,4) {};
\draw[order] (bot)--(1)--(a)--(0)--(top);
\node[label,anchor=east,xshift=1pt] at (top) {$0^2$};
\node[label,anchor=east,xshift=1pt] at (1) {$1$};
\node[label,anchor=east,xshift=1pt] at (a) {$a$};
\node[label,anchor=east,xshift=1pt] at (0) {$0$};
\node[label,anchor=east,xshift=1pt] at (bot) {$\bot$};
\node[label,anchor=north,yshift=-4pt] at (bot) {(vii)};
\end{scope}
\end{tikzpicture}
\caption{Some representable chains up to $5$ elements. 
Here all chains are commutative and $0^3 = 0^2\cdot 0 = 0\cdot 0^2$. In (vi) and (vii) we have $0^3=0\cdot a=0^2\cdot a = 0^2$. 
}\label{fig:representable_chains}
\end{figure}

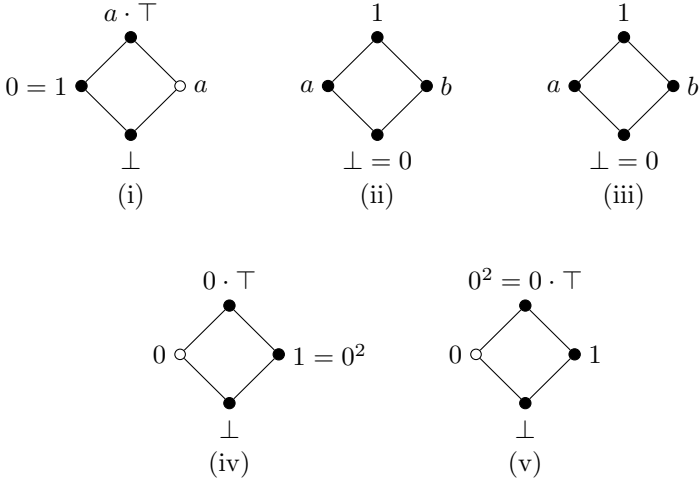
\begin{figure}[ht]
\centering
\begin{tikzpicture}[scale=0.65]
\begin{scope}
\node[draw,circle,inner sep=1.5pt,fill] (bot) at (0,0) {};
\node[draw,circle,inner sep=1.5pt,fill] (0) at (-1,1) {};
\node[unshaded] (a) at (1,1) {};
\node[draw,circle,inner sep=1.5pt,fill] (top) at (0,2) {};
\draw[order] (bot)--(0)--(top);
\draw[order] (bot)--(a)--(top);
\node[label,anchor=south] at (top) {$a\cdot \top$};
\node[label,anchor=east,xshift=1pt] at (0) {$0=1$};
\node[label,anchor=west,xshift=-1pt] at (a) {$a$};
\node[label,anchor=north] at (bot) {$\bot$};
\node[label,anchor=north,yshift=-12pt] at (bot) {(i)};
\end{scope}
\begin{scope}[xshift=5cm]
\node[draw,circle,inner sep=1.5pt,fill] (0) at (0,0) {};
\node[draw,circle,inner sep=1.5pt,fill] (a) at (-1,1) {};
\node[draw,circle,inner sep=1.5pt,fill] (b) at (1,1) {};
\node[draw,circle,inner sep=1.5pt,fill] (1) at (0,2) {};
\draw[order] (0)--(a)--(1);
\draw[order] (0)--(b)--(1);
\node[label,anchor=south] at (1) {$1$};
\node[label,anchor=east,xshift=1pt] at (a) {$a$};
\node[label,anchor=west,xshift=-1pt] at (b) {$b$};
\node[label,anchor=north] at (0) {$\bot = 0$};
\node[label,anchor=north,yshift=-12pt] at (0) {(ii)};
\end{scope}
\begin{scope}[xshift=10cm]
\node[draw,circle,inner sep=1.5pt,fill] (0) at (0,0) {};
\node[draw,circle,inner sep=1.5pt,fill] (a) at (-1,1) {};
\node[draw,circle,inner sep=1.5pt,fill] (b) at (1,1) {};
\node[draw,circle,inner sep=1.5pt,fill] (1) at (0,2) {};
\draw[order] (0)--(a)--(1);
\draw[order] (0)--(b)--(1);
\node[label,anchor=south] at (1) {$1$};
\node[label,anchor=east,xshift=1pt] at (a) {$a$};
\node[label,anchor=west,xshift=-1pt] at (b) {$b$};
\node[label,anchor=north] at (0) {$\bot = 0$};
\node[label,anchor=north,yshift=-12pt] at (0) {(iii)};
\end{scope}
\begin{scope}[xshift=2cm, yshift=-5.5cm]
\node[draw,circle,inner sep=1.5pt,fill] (bot) at (0,0) {};
\node[unshaded] (0) at (-1,1) {};
\node[draw,circle,inner sep=1.5pt,fill] (1) at (1,1) {};
\node[draw,circle,inner sep=1.5pt,fill] (top) at (0,2) {};
\draw[order] (bot)--(0)--(top);
\draw[order] (bot)--(1)--(top);
\node[label,anchor=south] at (top) {$0\cdot\top$};
\node[label,anchor=east,xshift=1pt] at (0) {$0$};
\node[label,anchor=west,xshift=-1pt] at (1) {$1=0^2$};
\node[label,anchor=north] at (bot) {$\bot$};
\node[label,anchor=north,yshift=-12pt] at (bot) {(iv)};
\end{scope}
\begin{scope}[xshift=8cm,yshift=-5.5cm]
\node[draw,circle,inner sep=1.5pt,fill] (bot) at (0,0) {};
\node[unshaded] (0) at (-1,1) {};
\node[draw,circle,inner sep=1.5pt,fill] (1) at (1,1) {};
\node[draw,circle,inner sep=1.5pt,fill] (top) at (0,2) {};
\draw[order] (bot)--(0)--(top);
\draw[order] (bot)--(1)--(top);
\node[label,anchor=south] at (top) {$0^2=0\cdot \top$};
\node[label,anchor=east,xshift=1pt] at (0) {$0$};
\node[label,anchor=west,xshift=-1pt] at (1) {$1$};
\node[label,anchor=north] at (bot) {$\bot$};
\node[label,anchor=north,yshift=-12pt] at (bot) {(v)};
\end{scope}
\end{tikzpicture}
\caption{Some $4$-element diamond representable quasi-relation algebras. 
All algebras listed here are commutative. In (i) we have $a^2 = 0 = 1=-1=1'$ and $-a = a' =a$. In (ii) we have $-a=b$, $-b = a$, $a'=a$ and $b'=b$, while in (iii) we have $-a = a'= b$ and $-b = b' = a$.  
}\label{fig:representable_diamonds}
\end{figure}

\begin{ex}\label{ex:chains1}
Algebras (i) and (iii) in Figure \ref{fig:representable_chains} and (i) in Figure \ref{fig:representable_diamonds} can be represented over $X = \left\{x, y\right\}$ with order ${\leqslant} = \text{id}_X$, $\alpha = \left\{\left(x, y\right), \left(y, x\right)\right\}$ and $\beta = \text{id}_X$ (or $\beta =\left\{\left(x, y\right), \left(y, x\right)\right\}$). In all three cases we take $E = X^2$. Then $0 = \alpha\mathbin{;} \left({\le}^c\right)^\smile = \left({\le}^c\right)^\smile\mathbin{;} \alpha = {\le}$. The identity $1$ maps to $\le$,  $\bot$ maps to $\varnothing$ and $\top$ to $X^2$, while $a$ in Figure \ref{fig:representable_diamonds}(i) maps to $\left\{\left(x,y\right), \left(y, x\right)\right\}$.  
\end{ex}

\begin{ex}
Algebras (ii) in Figure \ref{fig:representable_chains} and (iii) in Figure \ref{fig:representable_diamonds} can be represented over $X = \left\{x, y\right\}$ with $E = {\le} = \alpha = \beta = \text{id}_X$. In both cases, $1$ is mapped  ${\le} = E = X^2$ and $0$ is mapped to $\varnothing$. For (iii) in Figure \ref{fig:representable_diamonds}, $a$ is mapped to $\left\{\left(x, x\right)\right\}$ and $b$ to $\left\{\left(y, y\right)\right\}$. 
\end{ex}

\begin{ex}\label{ex:S4}
Algebra (iv) from Figure~\ref{fig:representable_chains} has the four-element Sugihara monoid as a term subreduct. It is  representable over 
$\left\langle \mathbb{Q}, \le\right\rangle$ (i.e., the rationals with the usual order) with $\alpha =\mathrm{id}_X$ and $\beta(x)=-x$. This is based on the work of Maddux~\cite{Mad2010} (see also~\cite{CR24}). It can be finitely represented by the X-shaped poset $X=\{v,w,x,y,z\}$ where $v,w<x$ and $x<y,z$. We use $\alpha=\{(v,w),(w,v),(x,x),(y,z),(z,y)\}$, $\beta=\{(y,v),(v,y),(x,x),(z,w),(w,z)\}$ and $E=X^2$. The embedding maps $1$ to $\leqslant$, $\bot$ to $\varnothing$, $\top$ to $X^2$ and $0$ to $\leqslant {\setminus} \{(x,x)\}$.
\end{ex}

\begin{ex}
The chain (v) in Figure \ref{fig:representable_chains} can be represented over $X = \left\{x, y, z\right\}$ with ${\le} = \text{id}_X$, $\alpha = \left\{\left(x,y\right), \left(y, z\right), \left(z,x\right)\right\}$, $\beta = \left\{\left(x, y\right), \left(y, x\right), \left(z,z\right)\right\}$ and $E = X^2$. The identity $1$ is mapped to $\le$, $\bot$ to $\varnothing$ and $\top$ to $E$. 
\end{ex}

\begin{ex}
Chain (vi) in Figure \ref{fig:representable_chains} can be represented over the set $X = \left\{w, x, y, z\right\}$ with ${\le} = \text{id}_X$, $E = X^2$, $\alpha = \left\{\left(w,x\right), \left(x, y\right), \left(y, z\right), \left(z, w\right)\right\}$ and $\beta = \left\{\left(w, x\right), \left(x, w\right), \left(y, z\right), \left(z, y\right)\right\}$. Here $\bot$ is mapped to $\varnothing$, $\top$ to $0^2 = E = X^2$ and $a$ to $\text{id}_X \cup \left\{\left(w, z\right), \left(x, w\right), \left(y, x\right), \left(z, y\right)\right\}$. 
For chain (vii) 
we use the same 
$(X,\leqslant)$, $E$ and $\beta$, but 
here we set 
$\alpha = \left\{\left(w, y\right), \left(y, w\right), \left(x, z\right), \left(z, x\right)\right\}$ and $a$ is mapped to $\text{id}_X \cup \left\{\left(w, x\right), \left(x, w\right), \left(y, z\right), \left(z, y\right)\right\}$.
\end{ex}

\begin{ex}
The diamond in Figure \ref{fig:representable_diamonds}(ii) can be represented over $X = \left\{w, x, y, z\right\}$ with order 
${\le} = \text{id}_X \cup \left\{\left(w,x\right), \left(w,y\right), \left(w,z\right), \left(x,z\right), \left(y, z\right)\right\}$, 
$\alpha= $ \break $\left\{\left(x,y\right), \left(y, x\right), \left(w,w\right), \left(z,z\right)\right\}$, $\beta = \left\{\left(x,x\right), \left(y, y\right), \left(w,z\right), \left(z,w\right)\right\}$ and $E = X^2$. Here $a$ is mapped to $\left\{\left(w,x\right), \left(w, y\right), \left(w,z\right),\left(x,x\right), \left(x,z\right), \left(y,y\right), \left(y,z\right), \left(z,z\right)\right\}$ and  $b$ to $\left\{\left(w, w\right),\left(w,x\right), \left(w, y\right), \left(w,z\right),\left(x,x\right), \left(x,z\right), \left(y,y\right), \left(y,z\right)\right\}$.
\end{ex}

\begin{ex}
The diamond in Figure \ref{fig:representable_diamonds}(iv) can be represented over $X = \left\{x, y\right\}$ with ${\le} = \alpha =\text{id}_X$, $\beta = \left\{\left(x, y\right), \left(y, x\right)\right\}$ and $E = X^2$. The identity $1$ is mapped to $\le$, $\bot$ is mapped to $\varnothing$, $\top$ to $E = X^2$ and $0$ to $\left(\le^c\right)^\smile = \left\{\left(x,y\right), \left(y, x\right)\right\}$. This construction also works with $\beta=\mathrm{id}_X$. 
\end{ex}

\begin{ex}
The diamond in Figure \ref{fig:representable_diamonds}(v) can be represented over $X = \left\{x, y,z\right\}$ with ${\le} = \alpha= \beta =\text{id}_X$ (or $\beta = \left\{\left(x, y\right), \left(y, x\right), \left(z, z\right)\right\}$) and $E = X^2$. 
\end{ex}

\begin{figure}[ht]
\centering
\begin{tikzpicture}[scale=0.65]
\begin{scope}
\node[draw,circle,inner sep=1.5pt,fill] (bot) at (0,0) {};
\node[unshaded] (0) at (0,1) {};
\node[draw,circle,inner sep=1.5pt,fill] (a) at (-1,2) {};
\node[draw,circle,inner sep=1.5pt,fill] (b) at (1,2) {};
\node[draw,circle,inner sep=1.5pt,fill] (1) at (0,3) {};
\node[draw,circle,inner sep=1.5pt,fill] (top) at (0,4) {};
\draw[order] (bot)--(0)--(a)--(1)--(top);
\draw[order] (bot)--(0)--(b)--(1)--(top);
\node[label,anchor=south] at (top) {$\top$};
\node[label,anchor=east,xshift=1pt] at (0) {$0$};
\node[label,anchor=east,xshift=1pt] at (a) {$a=a\cdot\top$};
\node[label,anchor=west,xshift=-1pt] at (b) {$b=\top\cdot b$};
\node[label,anchor=east,xshift=1pt] at (1) {$1$};
\node[label,anchor=north] at (bot) {$\bot = 0^2$};
\node[label,anchor=north,yshift=-12pt] at (bot) {(i)};
\end{scope}
\begin{scope}[xshift=6.75cm]
\node[draw,circle,inner sep=1.5pt,fill] (bot) at (0,0) {};
\node[draw,circle,inner sep=1.5pt,fill] (0) at (-1,1) {};
\node[draw,circle,inner sep=1.5pt,fill] (1) at (0,2) {};
\node[draw,circle,inner sep=1.5pt,fill] (a) at (1,1) {};
\node[draw,circle,inner sep=1.5pt,fill] (b) at (-2,2) {};
\node[draw,circle,inner sep=1.5pt,fill] (top) at (-1,3) {};
\draw[order] (bot)--(0)--(b)--(top);
\draw[order] (bot)--(a)--(1)--(top);
\draw[order] (bot)--(0)--(1)--(top);
\node[label,anchor=south] at (top) {$\top$};
\node[label,anchor=east,xshift=1pt] at (0) {$0$};
\node[label,anchor=west,xshift=-1pt] at (1) {$1$};
\node[label,anchor=west,xshift=-1pt] at (a) {$a$};
\node[label,anchor=east,xshift=1pt] at (b) {$b$};
\node[label,anchor=north] at (bot) {$\bot$};
\node[label,anchor=north,yshift=-12pt] at (bot) {(ii)};
\end{scope}
\begin{scope}[xshift=12cm]
\node[draw,circle,inner sep=1.5pt,fill] (bot) at (0,0) {};
\node[draw,circle,inner sep=1.5pt,fill] (b) at (0,1) {};
\node[draw,circle,inner sep=1.5pt,fill] (a) at (-1,1) {};
\node[unshaded] (0) at (1,1) {};
\node[draw,circle,inner sep=1.5pt,fill] (1) at (-1,2) {};
\node[draw,circle,inner sep=1.5pt,fill] (c) at (0,2) {};
\node[unshaded] (d) at (1,2) {};
\node[draw,circle,inner sep=1.5pt,fill] (top) at (0,3) {};
\draw[order] (bot)--(b)--(1)--(top);
\draw[order] (bot)--(b)--(d)--(top);
\draw[order] (bot)--(a)--(1)--(top);
\draw[order] (bot)--(a)--(c)--(top);
\draw[order] (bot)--(0)--(d)--(top);
\draw[order] (bot)--(0)--(c)--(top);
\node[label,anchor=south] at (top) {$\top \cdot d$};
\node[label,anchor=east,xshift=1pt] at (a) {$a=0^2$};
\node[label,anchor=east,xshift=1pt] at (b) {$b$};
\node[label,anchor=west,xshift=-1pt] at (0) {$0$};
\node[label,anchor=east,xshift=1pt] at (1) {$1=d^2$};
\node[label,anchor=west,xshift=-1pt] at (c) {$c$};
\node[label,anchor=west,xshift=-1pt] at (d) {$d$};
\node[label,anchor=north] at (bot) {$\bot$};
\node[label,anchor=north,yshift=-12pt] at (bot) {(iii)};
\end{scope}
\end{tikzpicture}
\caption{Other small representable cyclic distributive quasi relation algebras. In (i) we have $a\cdot b = a\cdot 0 = 0\cdot b = 0$, $\top \cdot a = b \cdot \top = \top$, $-a = b = b'$, $-b = a = a'$. Algebras (ii) and (iii) are commutative. In (ii) we have $a\cdot \top = a$, $0\cdot b= b\cdot \top = 0\cdot \top= b$,$-a = a' = b$ and $-b = b' = a$. Finally, in (iii) we have $0\cdot d = a$, $b\cdot d = b\cdot \top = b$, $a\cdot 0 = a\cdot d = 0$, $a\cdot c = a \cdot \top = c \cdot 0 = c \cdot \top = c \cdot d = c$, $-a = a' = d$ and $-b = b' = c$.}
\label{fig:bigger_examples}
\end{figure}
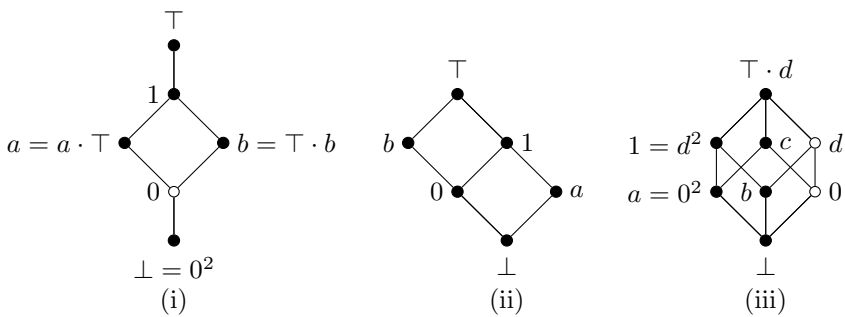

\begin{ex}
As we have seen in Examples \ref{ex:two-elmt-chain} and \ref{ex:prime-for-6elmt}, algebra (i) in Figure~\ref{fig:bigger_examples} can be represented over $X = \left\{x, y\right\}$ with  $\left\{\left(x,x\right), \left(x, y\right), \left(y, y\right)\right\}$, $E = X^2$, $\alpha = \text{id}_{X}$ and $\beta = \left\{\left(x, y\right), \left(y, x\right)\right\}$. 
\end{ex}

\begin{ex}
Both algebras (ii) and (iii) in Figure \ref{fig:bigger_examples} can be represented over $X = \left\{x, y, z\right\}$ with ${\le} = \text{id}_X$ and $E = \left\{\left(x, x\right), \left(y, y\right), \left(z, z\right), \left(x, y\right), \left(y, x\right)\right\}$. For (ii), we set $\alpha = \left\{\left(x, y\right), \left(y, x\right), \left(z, z\right)\right\}$ and $\beta = \text{id}_X$, and $\bot$ is mapped to $\varnothing$, $\top$ to $E$, $a$ to $\left\{\left(z, z\right)\right\}$ and $b$ to $\left\{\left(x, x\right), \left(y, y\right), \left(x, y\right), \left(y, x\right)\right\}$. For (iii), we set $\alpha = \text{id}_X$ and $\beta = \left\{\left(x, y\right), \left(y, x\right), \left(z, z\right)\right\}$, and $\bot$ is mapped to $\varnothing$, $\top$ to $E$, $a$ to $\left\{\left(x, x\right), \left(y, y\right)\right\}$ and $b$ to $\left\{\left(z, z\right)\right\}$.
\end{ex}

Our last example shows the importance of using a non-identity $\alpha$ in the construction. It is the  existence of a non-identity order automorphism that enables us to construct concrete non-cyclic DqRAs (see Theorem~\ref{thm:InFL-algebra}(ii)). 

\begin{ex}\label{ex:non-cyclic-concrete}
Let $X=\{x,y\}$ with ${\leqslant} = \{(x,x),(y,y)\}$, $\alpha = \{(x,y),(y,x)\}$ and $\beta = \text{id}_X$. Since the order on the poset is discrete, the underlying lattice structure of the full representable DqRA is that of a 16-element Boolean lattice. Consider $R=X^2{\setminus}\{(x,x)\}$. Clearly then $R^c=\{(x,x)\}$, and so $R'=$ \break $\alpha \mathbin{;} \beta \mathbin{;} R^c \mathbin{;} \beta = \alpha \mathbin{;} R^c = \{(y,x)\}$.
Hence, $'$ is not the complement and so $\mathbf{Dq}((X^2,\preceq))$ is not a relation algebra. However, note that this distributive quasi relation algebra contains the relation algebra (i) in Figure \ref{fig:representable_diamonds} as a subalgebra. 
\end{ex}

The following theorem gives a sufficient condition for a distributive quasi relation algebra to be not finitely representable. 
\begin{thm}\label{thm:non-rep}
Let $\mathbf A = \langle A, \wedge, \vee,\cdot, 1, 0, {\sim},{-}, '\rangle$ be a distributive quasi relation algebra. If there exists  $a \in A$ such that $0 < a < 1$ and $a\cdot a \leqslant 0$, then 
\begin{enumerate}[\normalfont (i)]
\item $\mathbf{A}$ is not representable over $(X,\leqslant)$ and $E$ with $\alpha$ having finite order; 
\item $\mathbf{A}$ is not finitely representable.
\end{enumerate} 
\end{thm}

\begin{proof}
We prove (i).
Let $\mathbf A$ be a distributive quasi relation algebra with $a\in A$ such that $0 < a < 1$ and $a\cdot a \leqslant 0$. Suppose for the sake of a contradiction that 
$\mathbf{A}$ is representable over $(X,\leqslant)$ and $E$ with 
$|\alpha|=n\in\mathbb{Z}^+$. 
Then there is an embedding $h: A \to \text{Up}\left(\mathbf E\right)$ such that $h(0) = \alpha\mathbin{;} \left(\le^c\right)^\smile$ and $h(1) = {\le}$. 

Since $h(0) \subset h(a)$, there is some $\left(x, y\right) \in E$ such that $\left(x, y\right) \in h(a)$ while $\left(x, y\right) \notin h(0)$. But $h(a) \subset h(1) = {\le}$, so we have $x \le y$. We claim that $x < y$. To see why, suppose $x = y$. Then $\left(x, y\right) \in h(a) \mathbin{;} h(a) \subseteq h(0)$, which is a contradiction. 

Since $(x,y) \notin h(0)=\alpha \mathbin{;} (\le^c)^\smile$, we have $(x,y) \in \alpha \mathbin{;} (\le^\smile)$ so $y \le \alpha(x)$. Now $x < y \le \alpha(x)$, and so, since $\alpha$ is an order automorphism, we obtain $x < \alpha(x) <\alpha(\alpha(x)) < \cdots < \alpha^n(x) = x$, a contradiction. 

Item (ii) follows from (i).
\end{proof}

\begin{ex}\label{ex:7elmt-non-rep} The seven-element non-cyclic DqRA on the left of Figure~\ref{fig:non-cyclic} has $c \cdot c=0$ and hence by Theorem~\ref{thm:non-rep}  is not finitely representable.
\end{ex}
\begin{ex}\label{ex:non-rep-chain}
Consider the three-element chain $0<a<1$ with $a\cdot a =0$. 
Clearly by Theorem~\ref{thm:non-rep} this is not finitely representable. 
Note that if it were non-representable, then Example~\ref{ex:7elmt-non-rep} would also be non-representable. 
Using Mace4 we have found a DqRA whose underlying lattice is a five-element chain $a<0<b<1<c$. It has  $b\cdot b =a$ and hence is also not finitely representable. 
\end{ex}

The only DqRAs with fewer than five elements
for which we do not have representations are 
Example~\ref{ex:non-rep-chain}, two four-element chains, and the four-element diamond which has 
$\bot < 0=1<\top$ and $\bot<c<\top$. The non-trivial products are $c\cdot c=c \cdot \bot = \bot \cdot c = \bot \cdot \bot =\bot$, $\top \cdot c = c \cdot \top = c$ and $\bot \cdot \top = \top \cdot \bot = \bot$. 

We have not been able to identify a non-representable DqRA. The most obvious candidates are the cyclic DqRAs 
that are term reducts of 
non-representable relation algebras. That is, where $\mathbf{A}=\langle A, \wedge, \vee, ', \bot, \top,\cdot, 1, ^\smile\rangle$, one considers the 
term reduct 
$d(\mathbf{A})=\langle A, \wedge, \vee, ', \cdot, 1, \sim,-\rangle$ where ${\sim}a={-}a=(a')^\smile$. 
We believe that an attempt to show that such algebras are not representable should include a study of the role of integrality.
Translating the concept of integrality from relation algebras to DqRAs would 
say that $1$ is join-irreducible (or completely join-irreducible if $\mathbf{A}$ is infinite). 

Notice that all of the algebras discussed in 
Example~\ref{ex:chains1} are cyclic, yet they are represented using $E=X^2$ with 
$\alpha\neq \text{id}_X$. This demonstrates that non-cyclic algebras of the form $\mathbf{Dq}(\mathbf{E})$ can be used to 
represent  cyclic DqRAs.
Proposition~\ref{prop:cyclic-elmts} makes clear why this is possible in general. 
This ability to represent cyclic DqRAs within larger non-cyclic algebras could possibly lead to representations of non-representable relation algebras.

\section{Conclusions}\label{sec:conclusions}

The main achievement of this paper is the construction given in Section~\ref{sec:construct}. Besides its application here, it could be used (without $\beta$, and $'$ given by complement-converse) to give  models of 
distributive DiDmFL$'$-algebras, distributive InDmFL$'$-algebras and distributive CyDiDmRL$'$-algebras.

Throughout the paper we have noted the connections between our construction and representability of so-called weakening relation algebras. 
Given any poset $(X,\leqslant)$, an algebra of weakening relations can be constructed. This was developed in the papers by Galatos and Jipsen~\cite{GJ20-AU, GJ20-ramics}, and has been recently explored by Jipsen and \v{S}emrl~\cite{JS23}. It is natural to wonder if one can embed any of these classes of representable algebras ($\mathsf{RWkRA}$, $\mathsf{RwkRA}$, $\mathsf{RDqRA}$) into any of the other classes. The fact that the three-element Sugihara chain is not an element of $\mathsf{RwkRA}$ shows that there cannot be an embedding from $\mathsf{RDqRA}$ into $\mathsf{RwkRA}$. Also, the fact that 
the full weakening relation algebra from the three-element poset in Figure~\ref{fig:ex_construction_50elements} 
is an element of both $\mathsf{RWkRA}$ and $\mathsf{RwkRA}$ indicates that neither of those classes can be embedded in $\mathsf{RDqRA}$.

The definition of representable distributive quasi relation algebras given in this paper immediately leads to a number of interesting open questions. 
\begin{enumerate}
\item Does the class $\RDqRA$ form a variety? 
Famously, $\mathsf{RRA}$ does form a variety. We commented at the start of Section~\ref{subsec:construct} that our construction resembles that used for $\RWkRA$ and $\RwkRA$. By~\cite[Corollary 6.2]{GJ20-AU}, $\RWkRA$ is a variety but in~\cite[Section 6]{JS23} it is shown that $\RwkRA$ is \emph{not} a variety. 
\item Are there non-representable DqRAs? As suggested in Section~\ref{sec:small-rep}, the easiest ones to find could be 
term reducts 
of non-representable relation algebras. 
\item Maddux~\cite{Mad78} defines the classes $\mathsf{K}$, $\mathsf{L}$ and $\mathsf{M}$ and thereby gives three neccessary conditions for representability. 
Can the conditions of Maddux be adapted to the setting of $\mathsf{RDqRA}$?
\item Is it possible to use games 
to define a hierarchy of classes of distributive quasi relation algebras $\DqRA_n$ similar to those defined by Maddux~\cite{Mad83} and Jipsen and \v{S}emrl~\cite{JS23}? In this case we would hope that $\DqRA_4$ would be the class of all DqRAs, and $\DqRA_\omega = \RDqRA$.
\item Is there a method for constructing qRAs of binary relations that are not necessarily distributive, but where the monoid operation is still given by relational composition? 
\end{enumerate}

\begin{figure}[ht!]
\centering
\begin{tikzpicture}[scale=0.49]
\begin{scope}
\node[draw,circle,inner sep=1.5pt,fill] (0) at (0,0) {};
\node[unshaded] (a) at (0,1.5) {};
\node[unshaded] (b) at (-4,3.5) {};
\node[unshaded] (c) at (-3,3.5) {};
\node[unshaded] (d) at (-2,3.5) {};
\node[unshaded] (e) at (-1,3.5) {};
\node[unshaded] (f) at (1,3.5) {};
\node[unshaded] (g) at (2,3.5) {};
\node[unshaded] (h) at (3,3.5) {};
\node[unshaded] (i) at (4,3.5) {};
\node[unshaded] (j) at (0,5.5) {};
\node[draw,circle,inner sep=1.5pt,fill] (1) at (0,7.5) {};
\draw[order] (0)--(a)--(b)--(j)--(1);
\draw[order] (0)--(a)--(c)--(j)--(1);
\draw[order] (0)--(a)--(d)--(j)--(1);
\draw[order] (0)--(a)--(e)--(j)--(1);
\draw[order] (0)--(a)--(f)--(j)--(1);
\draw[order] (0)--(a)--(g)--(j)--(1);
\draw[order] (0)--(a)--(h)--(j)--(1);
\draw[order] (0)--(a)--(i)--(j)--(1);
\node[label,anchor=east,xshift=1pt] at (0) {$0$};
\node[label,anchor=east,xshift=1pt,yshift=-2pt] at (a) {$a$};
\node[label,anchor=east,xshift=1pt,yshift=3pt] at (j) {$b$};
\node[label,anchor=east,xshift=2pt] at (1) {$1$};
\node[label,anchor=north,yshift=-1pt] at (0) {(i)};
\end{scope}

\begin{scope}[xshift=9cm]
\node[draw,circle,inner sep=1.5pt,fill] (0) at (0,0) {};
\node[unshaded] (a) at (0,1.5) {};
\node[unshaded] (b) at (-3,3) {};
\node[unshaded] (c) at (-1.5,3) {};
\node[unshaded] (d) at (1.5,3) {};
\node[unshaded] (e) at (3,3) {};
\node[unshaded] (f) at (-3,4.5) {};
\node[unshaded] (g) at (-1.5,4.5) {};
\node[unshaded] (h) at (1.5,4.5) {};
\node[unshaded] (i) at (3,4.5) {};
\node[unshaded] (j) at (0,6) {};
\node[draw,circle,inner sep=1.5pt,fill] (1) at (0,7.5) {};
\draw[order] (0)--(a)--(b)--(f)--(j)--(1);
\draw[order] (0)--(a)--(c)--(g)--(j)--(1);
\draw[order] (0)--(a)--(c)--(f)--(j)--(1);
\draw[order] (0)--(a)--(d);
\draw[order] (0)--(a)--(d)--(g)--(j)--(1);
\draw[order] (0)--(a)--(d)--(i)--(j)--(1);
\draw[order] (0)--(a)--(e)--(i)--(j)--(1);
\draw[order] (0)--(a)--(e)--(h)--(j)--(1);
\draw[order] (0)--(a)--(b)--(h)--(j)--(1);
\node[label,anchor=east,xshift=1pt] at (0) {$0$};
\node[label,anchor=east,xshift=1pt,yshift=-2pt] at (a) {$a$};
\node[label,anchor=east,xshift=1pt,yshift=3pt] at (j) {$b$};
\node[label,anchor=east,xshift=2pt] at (1) {$1$};
\node[label,anchor=east,xshift=1pt] at (b) {$c$};
\node[label,anchor=east,xshift=1pt] at (f) {${\sim}c$};
\node[label,anchor=west,xshift=-2pt] at (c) {${\sim}{\sim}c$};
\node[label,anchor=west,xshift=-0.5pt,yshift=2pt] at (g) {${\sim}^3 c$};
\node[label,anchor=west,xshift=-3.5pt,yshift=-0.4] at (h) {$-c$};
\node[label,anchor=west,xshift=-1.2pt] at (d) {$d$};
\node[label,anchor=west,xshift=-2pt] at (i) {${-}^3 c$};
\node[label,anchor=west,xshift=-2pt] at (e) {${-}{-}c$};
\node[label,anchor=north,yshift=-1pt] at (0) {(ii)};
\end{scope}
\end{tikzpicture}
\caption{Two 4-periodic qRAs. If we label the elements in the antichain of (i) from $c$ to $j$ from left to right, then $-c = i$, $-i = g$, $-g = d$, $-d = j$, ${\sim}c = h$, ${\sim}h = f$, ${\sim}f = e$, ${\sim}e = j$, $c' = c$, $d' =e$, $f'= g$, $h' = i$ and $j'=j$. For (ii) we have $c'= {\sim}^3c$, $d' = -c$ and $({--}c)' = {-}^3c$.} 
\label{fig:evenperiodic}
\end{figure}
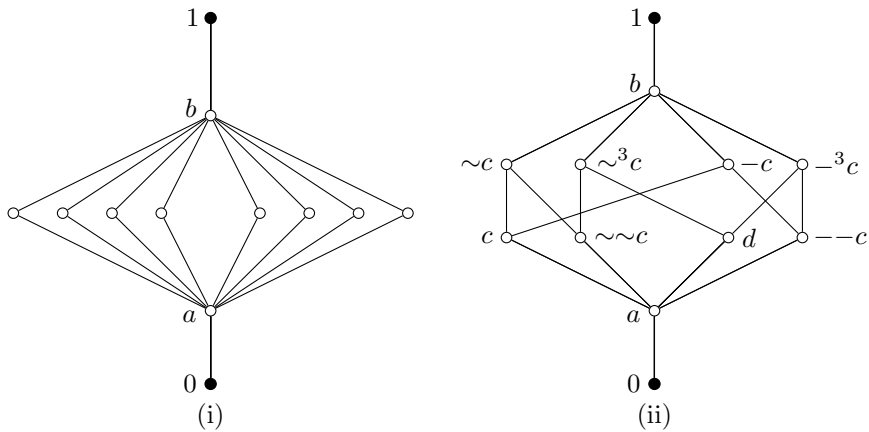

\vspace{-0.3cm}

Lastly, returning to the remarks at the end of Section~\ref{sec:qRAs}, we are also interested in the following:
\begin{enumerate}\setcounter{enumi}{5}
\item In Proposition~\ref{prop:odd-periodic-iso} we gave an isomorphism between qRAs
that are odd periodic. We would like to identify the conditions that allow such an isomorphism to exist for even periodic qRAs. The two examples in Figure~\ref{fig:evenperiodic} are both 4-periodic. 

For the algebra on the left, by letting 
$h(x)={\sim}x$ for all $x$ in the antichain and $h$ fixing $0,a,b,1$ we get an isomorphism between $\mathbf{A}$ and $\mathbf{A}^{\triangledown 1}$.  
For the algebra on the right, any order isomorphism $h$ must have 
$h({\sim}c)\in\{c,{\sim\sim} c,d,{--}c\}$. Testing each of these options reveals that such an $h$ will not be a qRA isomorphism from $\mathbf{A}$ to $\mathbf{A}^{\triangledown 1}$. 

Notice that neither lattice is distributive, although we do not think this is significant. High periodicity requires large antichains in the lattice, and wide distributive lattices have more elements than their non-distributive counterparts and hence are harder to find with Mace4. 
\end{enumerate}

\subsection*{Acknowledgements}

The authors would like to thank Peter Jipsen for many helpful discussions on this topic. We are grateful to the anonymous referee for their valuable comments. 

\section*{Declarations}

\subsection*{Ethical approval}
Not applicable. 

\subsection*{Competing interests} 
Not applicable. 

\subsection*{Authors' contributions} 
All authors contributed equally.

\subsection*{Availability of data and materials}
Not applicable.

\subsection*{Funding}
The first author acknowledges the support of the National Research Foundation (NRF) of South Africa (grant 127266).


\end{document}